\documentclass{amsart}
\usepackage{tikz}

\newcommand{\llll}[1]{\left\|{#1}\right\|}
\newcommand{\oo}[1]{\left({#1}\right)}
\newcommand{\cc}[1]{\left[{#1}\right]}
\newcommand{\ben}{\begin{eqnarray*}}
\newcommand{\een}{\end{eqnarray*}}

\newcommand{\norm}[1]{\left|{#1}\right|}

\newcommand{\intcurve}[1]{\int_{\Gamma}{{#1}\,ds}}
\newcommand{\intmink}[1]{\int_{\Gamma}{{#1}\,d\sigma}}
\newcommand{\kosc}{\mathscr{K}_{osc}}
\newcommand{\kave}{(\kappa-\bar{\kappa})}
\usepackage{amsmath}
\usepackage{amssymb}
\usepackage{url}
\usepackage{amscd}
\usepackage{mathrsfs}
\usepackage{hyperref}
\usepackage{mathtools}
\mathtoolsset{showonlyrefs,showmanualtags}
\usepackage[utf8]{inputenc}

\usepackage[margin=4cm]{geometry}
\allowdisplaybreaks

\newtheorem{thm}{Theorem}[section]
\newtheorem{prop}[thm]{Proposition}
\newtheorem{lem}[thm]{Lemma}
\newtheorem{cor}[thm]{Corollary}

\theoremstyle{remark}
\newtheorem{remark}{Remark}

\begin{document}


\title{The Anisotropic Polyharmonic Curve Flow for Closed Plane Curves}

\author[S. Parkins]{Scott Parkins}
\address{Institute for Mathematics and its Applications, School of Mathematics and Applied Statistics\\
University of Wollongong\\
Northfields Ave, Wollongong, NSW $2500$\\
Australia}
\email{srp854@uow.edu.au}

\author[G. Wheeler]{Glen Wheeler}
\address{Institute for Mathematics and its Applications, School of Mathematics and Applied Statistics\\
University of Wollongong\\
Northfields Ave, Wollongong, NSW $2500$\\
Australia}
\email{glenw@uow.edu.au}

\thanks{The research of the first author was supported by an Australian
Postgraduate Award. The research of the second author was partially supported
by Australian Research Council Discovery Project DP150100375.}


\begin{abstract}
We study the curve diffusion flow for closed curves immersed in the Minkowski plane $\mathcal{M}$, which is equivalent to the Euclidean plane endowed with a closed, symmetric, convex curve called an \emph{indicatrix} that scales the length of a vector in $\mathcal{M}$ depending on its length. The indiactrix $\partial\mathcal{U}$ (where $\mathcal{U}\subset\mathbb{R}^{2}$ is a convex, centrally symmetric domain) induces a second convex body, the isoperimetrix $\tilde{\mathcal{I}}$. This set is the unique convex set that miniminises the isoperimetric ratio (modulo homothetic rescaling) in the Minkowski plane. We prove that under the flow, closed curves that are initially close to a homothetic rescaling of the isoperimetrix in an averaged $L^{2}$ sense exists for all time and converge exponentially fast to a homothetic rescaling of the isoperimetrix that has enclosed area equal to the enclosed area of the initial immersion.

\end{abstract}

\maketitle


\newcommand\sfrac[2]{{#1/#2}}

\newcommand\cont{\operatorname{cont}}
\newcommand\diff{\operatorname{diff}}


\begin{section}{Introduction}\label{SectionIntroduction}
We consider a convex, centrally symmetric domain $\mathcal{U}\subset\mathbb{R}^{2}$ with symmetry centre $O$. We assume that $\partial\mathcal{U}$ is smooth with strictly positive Euclidean curvature. $\partial\mathcal{U}$ can be expressed as $\oo{r\oo{\theta}\cos\theta,r\oo{\theta}\sin\theta}$ with $\theta\in\left[0,2\pi\right)$ and $r>0$ and $r\oo{\theta+\pi}=r\oo{\theta}$.
For a vector $x\in\mathbb{R}^{2}$ with $x=\norm{x}\oo{\cos\theta,\sin\theta}$ (where $\norm{\cdot}$ is the regular Euclidean norm) the \emph{Minkowski norm}, $l\oo{x}$, of $x$ is defined by
\[
l\oo{x}=\frac{\norm{x}}{r\oo{\theta}}.
\]
Here $\theta=\theta\oo{x}$ is as defined earlier. Aptly, we then define the \emph{Minkowski plane} $\oo{\mathcal{M}^{2},d}$ with $\mathcal{M}^{2}=\mathbb{R}^{2}$ and $d$ the distance metric $d:\mathcal{M}^{2}\times\mathcal{M}^{2}\rightarrow\left[0,\infty\right)$ given by
\[
d\oo{x,y}=l\oo{x-y}.
\]
Hence for any $x\in\partial\mathcal{U}$ we have $d\oo{x,O}=l\oo{x}=1$, and so we define $\partial\mathcal{U}$ to be the \emph{Minkowski unit circle}, or \emph{indicatrix} of $\mathcal{M}^{2}$. The Minkowski plane is a vector space in which vector lengths are directionally-dependent. It carries its own notions of geometric quantities such as lengths and curvature which are reminiscent of their Euclidean counterparts. The indiactrix $\partial\mathcal{U}$ (where $\mathcal{U}\subset\mathbb{R}^{2}$ is a convex, centrally symmetric domain) induces a second convex body, the isoperimetrix $\tilde{\mathcal{I}}$. This set is the unique convex set that miniminises the isoperimetric ratio $\mathscr{I}$ (modulo homothetic rescaling) in the Minkowski plane (see Proposition \ref{IsoperimetrixProposition1}). The basics of convex body geometry, including the Minkowski plane, are introduced are introduced in greater detail in Section \ref{ConvexSection}.
\\\\
For $p\in\mathbb{N}$, we consider a one-parameter family of closed immersed curves $\Gamma:\mathbb{S}\times\left[0,T\right)\rightarrow\mathcal{M}^{2}$ with Minkowski normal velocity equal to $(-1)^{p}\kappa_{\sigma^{2p}}$:
\begin{equation}
\partial_{t}\Gamma(\sigma,t)=(-1)^{p}\kappa_{\sigma^{2p}}\cdot N(\sigma,t).\label{APH}\tag{$PF_{p}$}
\end{equation}
Here $\kappa_{\sigma^{2p}}=\frac{\partial^{2p}\kappa}{\partial\sigma^{2p}}$ refers to the $2p^{th}$ derivative of the Minkowski curvature with respect to the Minkowski arc length parameter (these concepts are introduced in depth in Section \ref{SectionIntroduction}). We will henceforth refer to a one-parameter family of closed curves evolving via \eqref{APH} as a \emph{$2\oo{p+1}$-anisotropic polyharmonic curve flow}, and it naturally generalises its lower-order Euclidean counterparts.
\\\\
If we consider the dual spaces to the Sobolev spaces $H^{p}$, which are denoted by $H^{-p}$ and consist of the bounded linear functionals $L:H^{p}\rightarrow\mathbb{R}$, 
then it can be shown that the anisotropic polyharmonic curves flows given by \eqref{APH} constitute a natural hierarchy of gradient flows for length in the Sobolev spaces $H^{-p}$. The $p^{th}$ step in the hierarchy corresponds to a flow of order $2(p+1)$. For example, the curve shortening flow (see, for example \cite{gage1983isoperimetric,Hamilton2,grayson1987heat} for the case of curves in the Euclidean plane and \cite{Gage1} for curves in the anisotropic setting that we are using today) is the $H^{-0}=L^{2}$ gradient flow for length, and the curve diffusion flow (see, for example \cite{Wheeler5,edwards2014shrinking}) is the $H^{-1}$ gradient flow for length. 
\\\\
We present the main theorem for this paper. It can be viewed as a higher-order
anisotropic analogue of the main results from \cite{parkins2015polyharmonic,Wheeler5}.

\begin{thm}\label{MainTheorem1}
Suppose that $\Gamma_{0}:\mathbb{S}^{1}\rightarrow\mathcal{M}^{2}$ is a regular
smooth immersed closed curve with $\mathscr{A}\oo{\Gamma_{0}}>0$.
Define $\mathscr{Q}:=h^{3}\oo{h+h_{\theta\theta}}$, where $h=r^{-1}$ is the
radial support function corresponding to the indicatrix $\partial\mathcal{U}$.
Then there exists a constant $\varepsilon_{0}>0$ such that if
\begin{equation}
\kosc(\Gamma_{0})<\varepsilon_{0}\,\,\text{and}\,\,\mathscr{I}\oo{\Gamma_{0}}<\exp\oo{\varepsilon_{0}/8\mathscr{A}(\tilde{\mathcal{I}})^{2}},\label{MainTheorem1,2}
\end{equation}
then the $2\oo{p+1}$-anisotropic polyharmonic curve flow
$\Gamma:\mathcal{S}^{1}\times\left[0,T\right)\rightarrow\mathcal{M}^{2}$ with
initial data $\Gamma\oo{\cdot,0}=\Gamma_{0}$ exists for all time and converges
exponentially fast to a homothetic rescaling of the isoperimetrix
$\tilde{\mathcal{I}}$ with enclosed area $\mathscr{A}\oo{\Gamma_{0}}$.
\end{thm}
In the above theorem, $\mathscr{I}$ $\kosc$ refers to the anisotropic normalised oscillation of curvature,
defined by
\[
\kosc(\Gamma) = \mathscr{L}\int_\Gamma (\kappa-\bar{\kappa})^2\,d\sigma
\,,
\]
which is introduced in Section \ref{KoscSection}, and $\mathscr{I}$ is the anisotropic isoperimetric ratio (see \eqref{IsoperimetricRatio2}).
\\\\
Any solution to \eqref{APH} with initial data satisfying the hypotheses of
Theorem \ref{MainTheorem1} yields a global solution that is asymptotic to a
homothetic rescaling of the isoperimetrix, i.e., that is eventually strictly
convex.
In this context convexity is defined by the positivity of the Minkowski
curvature $\kappa$ (although by the convexity of the isoperimetrix it is easy
to see that $\kappa>0$ if and only if $\Gamma$ is convex in the classical
sense).
It is therefore natural to ask how long one must wait until a solution
arising from Theorem \ref{MainTheorem1} becomes strictly convex.
The proposition below answers this question.

\begin{prop}[Upper bound on the waiting time for uniform conveity]\label{WaitingTimeProp}
Suppose that $\Gamma:\mathbb{S}^{1}\times[0,T)\rightarrow\mathcal{M}^{2}$ satisfies the criteria of Theorem \ref{MainTheorem1}. Then 
\[
\mathcal{L}\left\{t\in[0,\infty):\kappa(\cdot,t)\ngtr0\right\}\leq\frac{8\,\mathscr{A}(\tilde{\mathcal{I}})^{p-1}\mathscr{A}(\Gamma_{0})^{p+1}}{(p+1)\pi^{2p}}\oo{\mathscr{I}(\gamma_{0})^{p+1}-1},
\]
where $\kappa(\cdot,t)\ngtr0$ is taken to mean there exists at least one $\sigma_{0}$ with $\kappa(\sigma_{0},t)\leq0$.
\end{prop}
\end{section}

\begin{section}{Convex Bodies and Differential Geometry of the Minkowski Plane}\label{ConvexSection}
We introduce the fundamental concepts of convex body geometry before moving onto the Minkowski plane, the setting for our paper. To get a broader understanding of some of the finer details of the Minkowski plane and anisotropic vector spaces, the authors recommend reading the fantastic survey articles of Martini and Swanepoel \cite{martini2001geometry,martini2004geometry}.
\\\\
We begin with a real vector space $X$ and a proper subset $K\subset X$ containing the zero vector $\vec{0}$. The subset $K$ is assumed to possess the following properties:
\\
\begin{enumerate}
	\item [\textbf{(i)}] \label{Subset1} \emph{Convexity}: meaning that any convex combination of vectors in $\mathcal{U}$ is also contained in $\mathcal{U}$, and; \\
	\item [\textbf{(ii)}] \label{Subset2} \emph{Balanced}: meaning that $\alpha\,K\subseteq K$ for every $\norm{\alpha}\leq 1$.\\
	\item[]
\end{enumerate}
Such a subset $K$ is said to be \emph{absolutely convex}. If $K$ is absolutely convex, then we can define a corresponding \emph{Minkowskian functional} $p_{K}:X\rightarrow[0,\infty)$ by
\begin{equation}
p_{K}\oo{x}=\inf\left\{r>0:x\in r\,K\right\}.\label{MinkowskiFunctional}
\end{equation}
The properties \textbf{(i)} and \textbf{(ii)} prescribed to $K$ above allow us to ascertain that $p_{K}$ is \emph{subadditive}:
\begin{equation}
p_{K}\oo{x+y}\leq p_{K}\oo{x}+p_{K}\oo{y}\quad\forall x,y\in X\label{Subadditive}
\end{equation}
and \emph{homogeneous}:
\begin{equation}
p_{K}\oo{\alpha\,x}=|\alpha|\,p_{K}\oo{x}\quad\forall\alpha\in\mathbb{R}.\label{Homogeneous}
\end{equation}
Note that the properties \eqref{Subadditive} and \eqref{Homogeneous} imply that $p_{K}$ is a \emph{seminorm} for the vector space $X$.\\\\
As a simple example, consider the $n-$dimensional Euclidean vector space $X=\mathbb{R}^{n}$ and the closed ball centred at the origin of fixed radius $\rho>0$:
\[
K_{\rho}:=B_{\rho}\oo{0}:=\left\{x\in\mathbb{R}^{n}:|x|\leq \rho\right\}.
\] 
The set $K_{\rho}$ is obviously absolutely convex by the definition above. Moreover, for any $r>0$, one has 
\[
r\,K_{\rho}=B_{r\,\rho}\oo{0}=\left\{x\in\mathbb{R}^{n}:|x|\leq r\rho\right\},
\]
and so the associated Minkowski functional is easily calculable:
\[
p_{K_{\rho}}\oo{x}=\inf\left\{r>0:x\in r\,K_{\rho}\right\}=\inf\left\{r>0:x\in B_{r\,\rho}\oo{x}\right\}=\rho^{-1}|x|.
\]
Here $|\cdot|$ is the ordinary norm in $\mathbb{R}^{n}$. Therefore the Minkowski functional in this case simply scales a vector by a factor of $\rho^{-1}$. It is isotropic (meaning that vectors of the same Euclidean length map to the same value under $p_{K_{\rho}}$, independent of their direction). It is worth noting that if $\rho=1$ then $p_{K}$ simply gives the regular Euclidean vector length, $|\cdot|$. 
\\\\
Note that the ball centred at the origin is special in $\mathbb{R}^{n}$ in that it is invariant under all actions of $SO\oo{n}$, the special orthogonal group. This means that it is invariant under rotations, and therefore will induce a Minkowskian functional which is isotropic. For a generic absolutely convex body however this is certainly not the case, however, as you can quite clearly see by considering $K$ to be a non-spherical ellipsoid in $\mathbb{R}^{3}$ together with its interior. In this scenario a vector $x\in\mathbb{R}^{3}$ which is oriented in the direction of the longest semi-axis of $K$ attains a value of $p_{K}\oo{x}$ that is smaller than or equal to the value of $p_{K}$ evaluated at any proper rotation of $x$:
\[
p_{K}\oo{x}\leq p_{K}\oo{k\,\cdot x}\,\,\forall k\in SO\oo{3}.
\]
The Minkowskian functional in this case is \emph{anisotropic}, meaning that it is not invariant under rotations (that is, it is directionally dependent). This sets the scene for this part of the thesis. We now introduce the Minkowski plane $\mathcal{M}^{2}$, the anisotropic setting for our curve flow.
\\\\
We consider a convex, centrally symmetric domain $\mathcal{U}\subset\mathbb{R}^{2}$ with symmetry centre $\vec{0}$. We assume that $\partial\mathcal{U}$ is smooth with strictly positive Euclidean curvature. $\partial\mathcal{U}$ can be expressed as $\oo{r\oo{\theta}\cos\theta,r\oo{\theta}\sin\theta}$ with $\theta\in\left[0,2\pi\right)$ and $r>0$ and $r\oo{\theta+\pi}=r\oo{\theta}$.
For a vector $x\in\mathbb{R}^{2}$ with $x=|x|\oo{\cos\theta,\sin\theta}$ (where $\norm{\cdot}$ is the regular Euclidean norm) the \emph{Minkowski norm}, $l\oo{x}$, of $x$ is defined by
\begin{equation}
l\oo{x}=|x|/r\oo{\theta}.\label{MinkowskiNormDefn}
\end{equation}
Here $\theta=\theta\oo{x}$ is as defined earlier. One notes that because $\mathcal{U}$ is convex and centrally symmetric then it is automatically absolutely convex, and then the Minkowski norm $l$ satisfies the definition \eqref{MinkowskiFunctional} of the Minkowskian functional  corresponding to the body $K=\mathcal{U}$. Aptly, we then define the \emph{Minkowski plane}\footnote{One should be careful to not confuse this definition of the Minkowski plane with the other, perhaps more familiar notion  of the $2$- dimensional Minkowski spacetime, which is a $1+1$-dimensional Lorentzian manifold  which in local coordinates $\oo{t,x}$ is endowed with the metric $ds^{2}:=-dt^{2}+dx^{2}$.} $\oo{\mathcal{M}^{2},d}$ as the vector space $\mathcal{M}^{2}=\mathbb{R}^{2}$ equipped with the distance metric $d:\mathcal{M}^{2}\times\mathcal{M}^{2}\rightarrow\left[0,\infty\right)$ given by
\[
d\oo{x,y}=l\oo{x-y}.
\]
Hence for any $x\in\partial\mathcal{U}$ we have $d(x,\vec{0})=l\oo{x}=1$, and so we define $\partial\mathcal{U}$ to be the \emph{Minkowski unit circle}, or \emph{indicatrix} of $\mathcal{M}^{2}$. For example, if we are just working in $\mathbb{R}^{2}$ then our indicatrix $\partial\mathcal{U}$ is simply the Euclidean unit circle and the distance metric is the regular (isotropic) one defined by $d\oo{x,y}=\norm{x-y}$. Similarly, if $\partial\mathcal{U}$ is the Euclidean circle with radius $r$, then the corresponding distance metric is the one defined by $d\oo{x,y}=\norm{x-y}/r$. In this case the distance metric either isotropically enlarges or shrinks the length of vectors by a factor of $r$, depending upon whether $r<1$ or $r>1$, respectively.
\\\\
The \emph{polar dual} of $\mathcal{U}$, denoted $\mathcal{U}^{*}$ is given by
\[
\mathcal{U}^{*}:=\left\{f\in\oo{\mathcal{M}^{2}}^{*}:\norm{f\oo{x}}\leq 1\,\,\forall x\in\mathcal{U}\right\}\subset\oo{\mathcal{M}^{2}}^{*}.
\]
It is a simple exercise to show that if $\mathcal{U}$ is closed convex and contains the origin then $\mathcal{U}^{**}=\mathcal{U}$. This set is also a closed convex set (in the sense that convex combinations of linear functionals in $\oo{\partial\mathcal{U}}^{*}$ are also contained in the set). The boundary of the polar dual is given by
\[
\partial\mathcal{U}^{*}:=\left\{f\in\mathcal{U}^{*}:f\oo{x}=1\,\,\text{for some}\,\,x\in\partial\mathcal{U}\right\}.
\]

Recall that given a non-empty closed convex set $K\subset\mathbb{R}^{2}$, the \emph{support function} $h_{K}:\mathbb{R}^{2}\rightarrow\mathbb{R}$ of $K$ is given by
\[
h_{K}\oo{x}:=\sup\left\{\oo{x,k}:k\in K\right\},
\]
where $\oo{\cdot,\cdot}$ is the ordinary inner product in $\mathbb{R}^{2}$. If $K=\partial\mathcal{U}$ is parameterised by the angle function $\theta$ as before, then we define the \emph{polar radial support} function $h=h_{\mathcal{U}^{*}}$ as the support function of the polar dual of $\mathcal{U}$, $\mathcal{U}^{*}$. This function is also parameterised by $\theta$ and is in fact given by the reciprocal of the radial function: $h=r^{-1}$ \cite{Gage1}.
\\\\
Let $\Gamma:\mathbb{S}^{1}\rightarrow\mathcal{M}^{2}$ be a parametrised closed piecewise differentiable curve. Then we define the \emph{Minkowski length} of $\Gamma$ to be
\[
\mathscr{L}\oo{\Gamma}=\int_{\mathbb{S}^{1}}{l\oo{\Gamma_{u}}\,du}=\intmink{}.
\]
Here the Minkowski arc length element is given by
\begin{equation}
d\sigma\oo{u}=l\oo{\Gamma_{u}}du=\frac{\norm{\Gamma_{u}}}{r\oo{\Gamma_{u}}}du=\frac{ds\oo{u}}{r\oo{\Gamma_{u}}}.\label{MinkowskiArcLength}
\end{equation}
Alternatively, we write
\[
d\sigma\oo{s}=\frac{ds}{r\oo{\tau\oo{s}}}
\]
where $\tau=\Gamma_{s}$ is the Euclidean unit length tangent vector.
\\\\
Given our earlier parametrisation for $\partial\mathcal{U}$ we define the Minkowski tangent and Minkowski normal vector to a curve $\Gamma=\oo{x\oo{\theta},y\oo{\theta}}$ by
\begin{equation}
T\oo{\theta}=r\oo{\theta}\tau\oo{\theta}\quad\text{and}\quad N\oo{\theta}=-h_{\theta}\oo{\cos\theta,\sin\theta}+h\oo{-\sin\theta,\cos\theta},\label{MinkowskiTangentNormal}
\end{equation}
where $h$ is the \emph{polar radial support function}, $h=r^{-1}$. The derivation and reasoning behind these definitions are included at the beginning of the appendix. It should be noted that the angle $\theta$ refers to the angle that the regular Euclidean tangent to $\Gamma$ makes with the $x-$axis.
\\\\
The \emph{isoperimetrix} $\tilde{\mathcal{I}}$ is then defined by the parametrisation
\begin{equation}
\tilde{\mathcal{I}}=\left\{N\oo{\partial\mathcal{U}}\oo{\theta}:\theta\in\left[0,2\pi\right)\right\}
=\left\{-h_{\theta}\tau+hn:\theta\in\left[0,2\pi\right)\right\}.\label{IsoperimetrixDefinition}
\end{equation}
Qualitatively we have traced out the Minkowski normal vector $N$ as we vary along the indiactrix $\partial\mathcal{U}$.

\begin{prop}\label{IsoperimetrixProposition1}
For a Minkowski plane $\mathcal{M}^{2}$ with associated indicatrix $\partial\mathcal{U}$, a homothetic rescaling of $\tilde{\mathcal{I}}$ gives the minimum Minkowski boundary length of all convex sets with a given enclosed area.
\end{prop}
\begin{proof}
The result is standard. See, for example, \cite{chakerian1960isoperimetric}.
\end{proof}

Using \eqref{MinkowskiTangentNormal} as well as the chain rule and the identity $\theta_{s}=k$, we arrive at
\begin{equation}
\oo{
\begin{array}{c}
T\\
N
\end{array}
}_{\sigma}
=
\oo{
\begin{array}{cc}
0 & kh^{-3}\\
-k\oo{h+h_{\theta\theta}} & 0
\end{array}
}
\oo{
\begin{array}{c}
T\\
N
\end{array}
},\label{MinkowskiFrenet2}
\end{equation}
which is the anisotropic analogue of the standard Frenet–Serret formulas for planar curves (see, for example \cite{do1976differential}).\\\\

Let $\Gamma$ be a closed plane curve. Let $\theta$ be the angle between the
Minkowski tangent vector and the positive $x-$axis. That is to say,
$T=T\oo{\theta}$. It is relatively straightforward to check that the Euclidean
curvature of the isoperimetrix $\tilde{\mathcal{I}}$ at the point
$N\oo{\theta}$ is equal to
\begin{equation}
\hat{k}=\oo{h+h_{\theta\theta}}^{-1}.\label{EuclideanCurvature}
\end{equation}
By defining $T^{*},N^{*}$ to be the corresponding dual frames to $T$ and $N$
respectively (that is, $T^{*},N^{*}\in\oo{T\Gamma}^{*}$ satisfy
$T^{*}\oo{T}=N^{*}\oo{N}=1,T^{*}\oo{N}=N^{*}\oo{T}$ where $\oo{T\Gamma}^{*}$
denotes the cotangent bundle of $\Gamma$), one can show that the differential
of the Minkowski length functional, $dl$ can be expressed in a particularly
attractive way:
\[
dl=T^{*}.
\]
(For a calculation of this, see \cite{Gage1}).\\\\

It is then straightforward to show (see Remark \ref{Remark1}) that for a
one-parameter family of Minkowski immersions
$\Gamma:\mathbb{S}^{1}\times\left[0,T\right)\rightarrow\mathcal{M}^{2}$, the
evolution equation
\begin{equation}
\partial_{t}^{N}\Gamma=k\,\hat{k}^{-1}\,N=k\oo{h+h_{\theta\theta}}\,N=:\kappa\,N.\label{ACSF}\tag{ACSF}
\end{equation}
gives the flow of steepest descent for the Minkowski length functional (see Remark \ref{Remark1}). We call $\kappa=\kappa\oo{\sigma,t}$ the \emph{Minkowski curvature} associated to $\Gamma\oo{\sigma,t}$. Gage \cite{Gage1} (see also \cite{pozzi2012gradient} for another perspective) studied the motion
of a plane curve evolving with flow speed given by \eqref{ACSF}, the so-called
``anistropic curve-shortening flow''.  He proved that flows nomalised to
preserve area converge smoothly to a homothetic rescaling of
$\partial\mathcal{U}$ with enclosed area
$\mathscr{A}\oo{\Gamma_{0}}:=\mathscr{A}\oo{\Gamma\oo{\cdot,0}}$.  This is
clearly the Minkowski analogue of the regular Euclidean curve-shortening flow,
which has been studied quite thoroughly in the mathematical community (see
\cite{Abresch1,Altschuler1,Gage2,grayson1987heat}, among many others).
\\\\
In this paper we consider an anistropic polyharmonic curve flow of order $2\oo{p+1}$:
\begin{equation}
\partial_{t}\Gamma=\oo{-1}^{p}\kappa_{\sigma^{2p}}N.\label{APH}\tag{APH}
\end{equation} 
Here $\kappa_{\sigma^{2p}}=\frac{\partial^{2p}\kappa}{\partial \sigma^{2p}}$ refers to the $2p^{th}$ derivative of $\kappa$ with respect to $\sigma$. We will henceforth refer to a one-parameter family of closed curves evolving via \eqref{APH} as a \emph{$2\oo{p+1}$-anisotropic polyharmonic curve flow}, and it naturally generalises its lower-order Euclidean counterparts.
\\\\
We need an anistropic analogue to the Euclidean isoperimetric ratio $L^{2}/4\pi
A$, where $L,A$ are the Euclidean length and enclosed area, respectively.
Amongst closed curves in the Euclidean plane (so
$\mathcal{M}^{2}=\mathbb{R}^{2}$), this ratio is minimised by circles, in which
case it equals $1$. It turns out that for any closed curve immersed in the
Minkowski plane $\mathcal{M}^{2}$,
$\Gamma:\mathbb{S}^{1}\rightarrow\mathcal{M}^{2}$ with positive enclosed area,
the following inequality holds:
\begin{equation}
\mathscr{L}^{2}\oo{\Gamma}-2\mathscr{A}\oo{\Gamma}\intmink{\kappa}\geq0,\label{IsoperimetricRatio1}
\end{equation}
with equality if and only if $\Gamma$ is a homothetic rescaling of the isoperimetrix $\tilde{\mathcal{I}}$. Noting that $\intmink{\kappa}=2\,\mathscr{A}(\tilde{\mathcal{I}})$ (see Proposition \ref{TotalCurvatureProp}), the ratio
\begin{equation}
\mathscr{I}\oo{\Gamma}:=\frac{\mathscr{L}^{2}\oo{\Gamma}}{4\mathscr{A}\oo{\Gamma}\mathscr{A}(\tilde{\mathcal{I}})}\label{IsoperimetricRatio2}
\end{equation}
is called the \emph{anisotropic isoperimetric ratio} associated to $\mathcal{M}^{2}$. This ratio is always greater than or equal to $1$, with equality if an only if $\Gamma$ is a homothetic rescaling of the isoperimetrix $\tilde{\mathcal{I}}$ (hence its name). In this paper we will often refer to $\mathscr{I}$ simply as the ``isoperimetric ratio'' for short, since it is in fact equal to its Euclidean counterpart in the case $\mathcal{M}^{2}=\mathbb{R}^{2}$.\\\\

Throughout this paper we will frequently use what is referred to as the \emph{$P-$style notation} for brevity. For a function $\varphi$ defined on the curve $\Gamma$ we use the notation
\[
P_{i}^{j}(\varphi):=\sum_{r_{1}+\dots+r_{j}=i}c\,\partial_{\sigma}^{r_{1}}\varphi\cdot\partial_{\sigma}^{r_{2}}\varphi\dots\partial_{\sigma}^{r_{j}}\varphi
\]
where the constant $c$ may vary from one term in the summation to another. Sometimes we are also interested in the highest derivative that appears in the summation in question. In this case we write
\[
P_{i}^{j,k}(\varphi):=\sum_{r_{1}+\dots+r_{j}=i,\,r_{l}\leq k}c\,\partial_{\sigma}^{r_{1}}\varphi\cdot\partial_{\sigma}^{r_{2}}\varphi\dots\partial_{\sigma}^{r_{j}}\varphi.
\]
Note that all the derivatives of $\varphi$ in our summation are less than or equal to $k$. This will be especially important when we wish to apply the inequalities from Lemma \ref{AppendixLemma4} and Lemma \ref{AppendixLemma5} later on.
\\\\

The paper is organised as follows.
In Section 2 we calculate evolution equations of various integrals.
Section 3 is concerned with the anisotropic oscillation of curvature.
We prove an estimate for $\kosc(\Gamma)$ that allows us to give a
charactisation of finite-time blowup.
The characterisation is blowup of the $L^2$-anisotropic norm of curvature, with
rate dependent on the order of the flow.
Section 4 considers the global behaviour of the flow, beginning with an
observation that the characterisation of finite-time blowup allows us to
conclude global existence if the anisotropic oscillation of curvature is initially 
smaller than an explicit constant, as in this case our estimates uniformly
control the $L^2$-anisotropic norm of curvature along the flow.

\section{Evolution equations}

\begin{lem}\label{Lemma1}
Suppose that $\Gamma:\mathbb{S}^{1}\times\left[0,T\right)\rightarrow\mathcal{M}^{2}$ solves \eqref{APH}, and that $f:\mathbb{S}^{1}\times\left[0,T\right)\rightarrow\mathbb{R}$ is a periodic function with the same period as $\Gamma$. Then
\[
\frac{d}{dt}\intmink{f}=\intmink{f_{t}}+\intmink{f\partial_{t}}=\intmink{f_{t}+\oo{-1}^{p+1}f\cdot\kappa\cdot\kappa_{\sigma^{2p}}}.
\]
\end{lem}
\begin{proof}
We must first calculate the time derivaties of the Euclidean and Minkowski length elements, $ds$ and $d\sigma$ respectively. Firstly, 
\begin{align}
\partial_{t}\norm{\Gamma_{u}}^{2}&=2(\Gamma_{u},\partial_{u}\Gamma_{t})\\
&=2\oo{-1}^{p}(\Gamma_{u},\partial_{u}\oo{\kappa_{\sigma^{2p}}N})\nonumber\\
&=2\oo{-1}^{p}h\norm{\Gamma_{u}}^{2}(\tau,\partial_{\sigma}\oo{\kappa_{\sigma^{2p}}N})\\
&=2\oo{-1}^{p}h\norm{\Gamma_{u}}^{2}(\tau,\kappa_{\sigma^{2p+1}}N-\kappa\cdot\kappa_{\sigma^{2p}}T)\nonumber\\
&=2\oo{-1}^{p}h\norm{\Gamma_{u}}^{2}(\tau,\kappa_{\sigma^{2p+1}}\oo{-h_{\theta}\tau+hn}-r\kappa\cdot\kappa_{\sigma^{2p}}\tau)\nonumber\\
&=2\oo{-1}^{p+1}\norm{\Gamma_{u}}^{2}\oo{h\cdot h_{\theta}\kappa_{\sigma^{2p+1}}+\kappa\cdot\kappa_{\sigma^{2p}}}.\nonumber
\end{align}
It follows immediately that
\begin{equation}
\partial_{t}ds=\partial_{t}\oo{\norm{\Gamma_{u}}du}=\oo{-1}^{p+1}\oo{h\cdot h_{\theta}\kappa_{\sigma^{2p+1}}+\kappa\cdot\kappa_{\sigma^{2p}}}\,ds.\label{EuclideanArclengthEvolution}
\end{equation}
Next, by using the chain rule, it is relatively straightforward to calculate
\begin{align}
\partial_{t}r\oo{\theta}&=\partial_{t}r\oo{\tan^{-1}(y_{u}/x_{u})}\\
&=\frac{\partial r\oo{\tan^{-1}(y_{u}/x_{u})}}{\partial \oo{\tan^{-1}(y_{u}/x_{u})}}\frac{\partial \oo{\tan^{-1}(y_{u}/x_{u})}}{\partial(y_{u}/x_{u})}\frac{\partial(y_{u}/x_{u})}{\partial t}\nonumber\\
&=\frac{r_{\theta}}{\norm{\Gamma_{u}}^{2}}(\oo{-y_{u},x_{u}},\partial_{u}\Gamma_{t})\\
&=r_{\theta}(n,\partial_{s}\Gamma_{t})\\
&=\oo{-1}^{p}hr_{\theta}(n,\partial_{\sigma}\oo{\kappa_{\sigma^{2p}}N})\nonumber\\
&=\oo{-1}^{p-1}h_{\theta}\kappa_{\sigma^{2p+1}}.\label{Lemma1,1}
\end{align}
Combining \eqref{EuclideanArclengthEvolution} and \eqref{Lemma1,1} allows us to calculate the evolution of Minkowski length element:
\begin{align}
\partial_{t}d\sigma&=-r^{-2}\oo{\oo{-1}^{p-1}h_{\theta}\kappa_{\sigma^{2p+1}}}\,ds+r^{-1}\oo{-1}^{p+1}\oo{h\cdot h_{\theta}\kappa_{\sigma^{2p+1}}+\kappa\cdot\kappa_{\sigma^{2p}}}\,ds\nonumber\\
&=\oo{-1}^{p+1}r^{-1}\kappa\cdot\kappa_{\sigma^{2p}}\,ds\\
&=\oo{-1}^{p+1}\kappa\cdot\kappa_{\sigma^{2p}}\,d\sigma.\label{MinkowskiArclengthEvolution}
\end{align}
It follows that if $f:\mathbb{S}^{1}\times\left[0,T\right)\rightarrow\mathbb{R}$ follows the hypothesis of the lemma, then
\[
\frac{d}{dt}\intmink{f}=\intmink{f_{t}}+\intmink{f\partial_{t}}=\intmink{f_{t}+\oo{-1}^{p+1}f\cdot\kappa\cdot\kappa_{\sigma^{2p}}},
\]
which is the desired result.
\end{proof}
\begin{cor}\label{Corollary1}
Suppose that $\Gamma:\mathbb{S}^{1}\times\left[0,T\right)\rightarrow\mathcal{M}^{2}$ solves \eqref{APH}. Then
\[
\frac{d}{dt}\mathscr{L}\oo{\Gamma}=-\intmink{\kappa_{\sigma^{p}}^{2}}\leq0\,\,\text{and}\,\,\frac{d}{dt}\mathscr{A}\oo{\Gamma}=0.
\]
Moreover,
\[
\frac{d}{dt}\intmink{\kappa}=0\,\,\text{nd}\,\,\frac{d}{dt}\overline{\kappa}\geq0,
\]
where $\overline{f}$ refers to the average of a function $f$ over $\gamma$:
\[
\overline{f}:=\frac{1}{\mathscr{L}}\intmink{f}.
\]
As a result, the isoperimetric ratio $\mathscr{I}$ decreases monotonically along the flow, with
\[
\mathscr{I}(t)\leq\mathscr{I}(0)\exp\oo{-\frac{2\int_{0}^{t}{||\kappa_{\sigma^{p}}||_{2}^{2}\,d\tau}}{\mathscr{L}(0)}}.
\]
\end{cor}
\begin{proof}
Applying Lemma \ref{Lemma1} with $f\equiv1$ and integrating by parts $p$ times gives the first claim of the corollary immediately:
\[
\frac{d}{dt}\mathscr{L}=\oo{-1}^{p+1}\intmink{\kappa\cdot\kappa_{\sigma^{2p}}}=-\intmink{\kappa_{\sigma^{p}}^{2}}\oo{\leq0}.
\]
For the second claim of the corollary, we employ equation \eqref{EuclideanArclengthEvolution}:
\begin{equation}
\frac{d}{dt}\intcurve{(\Gamma,n)}=\intcurve{\partial_{t}(\Gamma,n)+\oo{-1}^{p+1}(\Gamma,n)\oo{h\cdot h_{\theta}\kappa_{\sigma^{2p+1}}+\kappa\cdot\kappa_{\sigma^{2p}}}}.\label{Corollary1,1}
\end{equation}
All that is left is to calculate $n_{t}$. To do so, we will need to first calculate the commutator $\cc{\partial_{t},\partial_{s}}$. Some of the work from Lemma \ref{Lemma1} will assist us along the way. We have
\begin{align}
\partial_{ts}&=\partial_{t}\oo{\norm{\Gamma_{u}}^{-1}\partial_{u}}=-\norm{\Gamma_{u}}^{-2}\partial_{t}\norm{\Gamma_{u}}\partial_{u}+\partial_{st}\nonumber\\
&=\oo{-1}^{p}\norm{\Gamma_{u}}^{2}\oo{h\cdot h_{\theta}\kappa_{\sigma^{2p+1}}+\kappa\cdot\kappa_{\sigma^{2p}}}\norm{\Gamma_{u}}\,\partial_{u}+\partial_{st}\nonumber\\
&=\partial_{st}+\oo{-1}^{p}\oo{h\cdot h_{\theta}\kappa_{\sigma^{2p+1}}+\kappa\cdot\kappa_{\sigma^{2p}}}\partial_{s}.\label{Commutatorst}
\end{align}
Next, note that $\tau=\Gamma_{s}$, and so the identity \eqref{Commutatorst} implies that
\begin{align}
\partial_{t}\tau&=\partial_{st}\Gamma+\oo{-1}^{p}\oo{h\cdot h_{\theta}\kappa_{\sigma^{2p+1}}+\kappa\cdot\kappa_{\sigma^{2p}}}\partial_{s}\Gamma\nonumber\\
&=\oo{-1}^{p}\partial_{s}\oo{\kappa_{\sigma^{2p}}N}+\oo{-1}^{p}\oo{h\cdot h_{\theta}\kappa_{\sigma^{2p+1}}+\kappa\cdot\kappa_{\sigma^{2p}}}\tau\nonumber\\
&=\oo{-1}^{p}h\oo{\kappa_{\sigma^{2p+1}}N-\kappa\cdot\kappa_{\sigma^{2p}}T}+\oo{-1}^{p}\oo{h\cdot h_{\theta}\kappa_{\sigma^{2p+1}}+\kappa\cdot\kappa_{\sigma^{2p}}}\tau\nonumber\\
&=\oo{-1}^{p}h\kappa_{\sigma^{2p+1}}\oo{hn-h_{\theta}\tau}+\oo{-1}^{p+1}\kappa\cdot\kappa_{\sigma^{2p}}\tau+\oo{-1}^{p}\oo{h\cdot h_{\theta}\kappa_{\sigma^{2p+1}}+\kappa\cdot\kappa_{\sigma^{2p}}}\tau\nonumber\\
&=\oo{-1}^{p}h^{2}\kappa_{\sigma^{2p+1}}n.\label{Corollary1,3}
\end{align}
Next, the identity $\norm{n}^{2}=1$ implies that $\partial_{t}n\perp n$. Hence from \eqref{Corollary1,3} we have
\begin{equation}
\partial_{t}n=(\partial_{t}n,\tau)\tau=-(n,\partial_{t}\tau)\tau=\oo{-1}^{p+1}h^{2}\kappa_{\sigma^{2p+1}}\tau\label{Corollary1,4}
\end{equation}
Hence
\begin{multline}
\partial_{t}(\Gamma,n)=\oo{-1}^{p}(\kappa_{\sigma^{2p}}N,n)+\oo{-1}^{p+1}h^{2}\kappa_{\sigma^{2p+1}}(\Gamma,\tau)\\
=\oo{-1}^{p}h\kappa_{\sigma^{2p}}+\oo{-1}^{p+1}h^{2}\kappa_{\sigma^{2p+1}}(\Gamma,\tau).
\end{multline}
Substituting this back into \eqref{Corollary1,1} gives
\begin{multline}
\frac{d}{dt}\intcurve{(\Gamma,n)}=\oo{-1}^{p}\intmink{\kappa_{\sigma^{2p}}}\\
+\oo{-1}^{p+1}\intcurve{h^{2}\kappa_{\sigma^{2p+1}}(\Gamma,\tau)}+\oo{-1}^{p+1}\intcurve{(\Gamma,n)\oo{h\cdot h_{\theta}\kappa_{\sigma^{2p+1}}+\kappa\cdot\kappa_{\sigma^{2p}}}}.\label{Corollary1,5}
\end{multline}
For the second last term in \eqref{Corollary1,5} we apply integration by parts once:
\begin{align}
&\oo{-1}^{p+1}\intcurve{(\Gamma,n)h\cdot h_{\theta}\kappa_{\sigma^{2p+1}}}\nonumber\\
&=\oo{-1}^{p+1}\intcurve{(\Gamma,n)h_{\theta}\kappa_{\sigma^{2p}s}}\nonumber\\
&=\oo{-1}^{p}\intcurve{\kappa_{\sigma^{2p}}\oo{(\Gamma_{s},n)h_{\theta}+(\Gamma,n_{s})h_{\theta}(\Gamma,n)h_{\theta s}}}\nonumber\\
&=\oo{-1}^{p+1}\intcurve{k\cdot\kappa_{\sigma^{2p}}h_{\theta}(\Gamma,\tau)}+\oo{-1}^{p}\intcurve{k\cdot h_{\theta\theta}\kappa_{\sigma^{2p}}(\Gamma,n)}.\label{Corollary1,6}
\end{align}
Here we have also used the identity $\partial_{s}=k\partial_{\theta}$ and the fact that $\partial_{s}\Gamma=\tau\perp n$. Simiarly, for the third last term in \eqref{Corollary1,5} we use integration by parts once more:
\begin{align}
&\oo{-1}^{p+1}\intcurve{h^{2}\kappa_{\sigma^{2p+1}}(\Gamma,\tau)}\nonumber\\
&=\oo{-1}^{p+1}\intcurve{h\kappa_{\sigma^{2p}s}(\Gamma,\tau)}\nonumber\\
&=\oo{-1}^{p}\intcurve{\kappa_{\sigma^{2p}}\oo{h_{s}(\Gamma,\tau)+h(\Gamma_{s},\tau)+h(\Gamma,\tau_{s})}}\nonumber\\
&=\oo{-1}^{p}\intcurve{k\cdot\kappa_{\sigma^{2p}}h_{\theta}(\Gamma,\tau)}+\oo{-1}^{p}\intmink{\kappa_{\sigma^{2p}}}\nonumber\\
&+\oo{-1}^{p}\intcurve{k\cdot h\kappa_{\sigma^{2p}}(\Gamma,n)\cdot}.\label{Corollary1,7}
\end{align}
Combining \eqref{Corollary1,6},\eqref{Corollary1,7} and substituting back into \eqref{Corollary1,5} gives
\begin{align}
\frac{d}{dt}\mathscr{A}&=-\frac{1}{2}\frac{d}{dt}\intcurve{(\Gamma,n)}\nonumber\\
&=\oo{-1}^{p+1}\intmink{\kappa_{\sigma^{2p}}}+\frac{1}{2}\oo{-1}^{p+1}\intcurve{k\oo{h+h_{\theta\theta}}\kappa_{\sigma^{2p}}(\Gamma,n)}\nonumber\\
&\quad+\frac{1}{2}\oo{-1}^{p}\intcurve{\kappa\cdot\kappa_{\sigma^{2p}}(\Gamma,n)}\nonumber\\
&=\oo{-1}^{p+1}\intmink{\kappa_{\sigma^{2p}}}=0.\nonumber
\end{align}
Here we have used $\kappa=k\oo{h+h_{\theta\theta}}$. The last step follows from the divergence theorem because $\Gamma\oo{\cdot,t}$ is closed. This establishes the second claim of the Corollary. 
\\\\
We first calculate the evolution of the Euclidean curvature $k$. This is relatively simple:
\begin{align}
\partial_{t}k&=\partial_{t}(\gamma_{ss},n)=(\partial_{ts}\tau,n)+(\tau_{s},\partial_{t}n)\nonumber\\
&=(\partial_{ts}\tau,n)=(\partial_{st}\tau+\oo{-1}^{p}\oo{h\cdot h_{\theta}\kappa_{\sigma^{2p+1}}+\kappa\cdot\kappa_{\sigma^{2p}}}\tau_{s},n)\nonumber\\
&=(\partial_{s}\oo{\oo{-1}^{p}h^{2}\kappa_{\sigma^{2p+1}}n},n)+\oo{-1}^{p}k\oo{h\cdot h_{\theta}\kappa_{\sigma^{2p+1}}+\kappa\cdot\kappa_{\sigma^{2p}}}\nonumber\\
&=\oo{-1}^{p}\oo{2kh\cdot h_{\theta}\kappa_{\sigma^{2p+1}}+h^{3}\kappa_{\sigma^{2p+2}}}+\oo{-1}^{p}k\oo{h\cdot h_{\theta}\kappa_{\sigma^{2p+1}}+\kappa\cdot\kappa_{\sigma^{2p}}}\nonumber\\
&=\oo{-1}^{p}\oo{h^{3}\kappa_{\sigma^{2p+2}}+3kh\cdot h_{\theta}\kappa_{\sigma^{2p+1}}+k\kappa\cdot\kappa_{\sigma^{2p}}}.\label{Lemma2,1}
\end{align}
Here we have used $\tau_{s}\perp\partial_{t}n$. Next we need to calculate the evolution of $h,h_{\theta\theta}$. This turns out to be relatively straightforward. For any $m\in\mathbb{Z}_{0}$ a simple calculation gives
\begin{align}
\partial_{t}h_{\theta^{m}}&=\frac{1}{\norm{\Gamma_{u}}}(\tau,\partial_{t}\oo{y_{u},-x_{u}})h_{\theta^{m+1}}=-(\tau,n_{t})h_{\theta^{m+1}}\nonumber\\
&=\oo{-1}^{p}h^{2}h_{\theta^{m+1}}\kappa_{\sigma^{2p+1}}.\label{Lemma2,1,2}
\end{align}
Hence
\begin{align}
\partial_{t}\oo{h+h_{\theta\theta}}&=\oo{-1}^{p}h^{2}\oo{h+h_{\theta\theta}}_{\theta}\kappa_{\sigma^{2p+1}}\nonumber\\
&=\oo{-1}^{p}k^{-1}h^{3}\oo{h+h_{\theta\theta}}_{\sigma}\kappa_{\sigma^{2p+1}}.\label{Lemma2,2}
\end{align}
Combining \eqref{Lemma2,1} and \eqref{Lemma2,2} gives us
\begin{align}
\partial_{t}\kappa&=\partial_{t}\oo{k\oo{h+h_{\theta\theta}}}\nonumber\\
&=\oo{-1}^{p}\oo{h^{3}\kappa_{\sigma^{2p+2}}+3kh\cdot h_{\theta}\kappa_{\sigma^{2p+1}}+k\kappa\cdot\kappa_{\sigma^{2p}}}\oo{h+h_{\theta\theta}}\nonumber\\
&\quad+\oo{-1}^{p}h^{3}\oo{h+h_{\theta\theta}}_{\sigma}\kappa_{\sigma^{2p+1}}\nonumber\\
&=\oo{-1}^{p}\oo{h^{3}\oo{h+h_{\theta\theta}}\kappa_{\sigma^{2p+1}}}_{\sigma}+\oo{-1}^{p}\kappa^{2}\cdot\kappa_{\sigma^{2p}}.\label{Lemma2,3}
\end{align}
Applying Lemma \ref{Lemma1} then yields
\begin{align*}
\frac{d}{dt}\intmink{\kappa}&=\intmink{\partial_{t}\kappa}+\oo{-1}^{p+1}\intmink{\kappa^{2}\cdot\kappa_{\sigma^{2p}}}\nonumber\\
&=\oo{-1}^{p}\intmink{\oo{h^{3}\oo{h+h_{\theta\theta}}\kappa_{\sigma^{2p+1}}}_{\sigma}}+\oo{-1}^{p}\intmink{\kappa^{2}\cdot\kappa_{\sigma^{2p}}}\\
&\quad+\oo{-1}^{p+1}\intmink{\kappa^{2}\cdot\kappa_{\sigma^{2p}}}\\
&=\oo{-1}^{p}\intmink{\oo{h^{3}\oo{h+h_{\theta\theta}}\kappa_{\sigma^{2p+1}}}_{\sigma}}=0.
\end{align*}
Here the last step follows from the divergence theorem because $\Gamma\oo{\cdot,t}$ is closed. This gives the first claim of the lemma. The second claim follows immediately from combining the first claim and the first claim of Corollary \ref{Corollary1}:
\[
\frac{d}{dt}\bar{\kappa}=-\mathscr{L}^{-2}\intmink{\kappa}\cdot\frac{d}{dt}\mathscr{L}=\frac{2\mathscr{A}(\tilde{\mathcal{I}})}{\mathscr{L}^{2}}\intmink{\kappa_{\sigma^{p}}^{2}}\geq0.
\]
For the final claim regarding the isoperimetric ratio, simply combine the previous two results:
\begin{equation}
\frac{d}{dt}\mathscr{I}=\frac{d}{dt}\oo{\frac{\mathscr{L}^{2}}{4\mathscr{A}(\tilde{\mathcal{I}})\,\mathscr{A}(\Gamma)}}=-\frac{\mathscr{L}\intmink{\kappa_{\sigma^{p}}^{2}}}{2\mathscr{A}(\tilde{\mathcal{I}})\,\mathscr{A}(\Gamma)}\leq-\frac{2\mathscr{I}\intmink{\kappa_{\sigma^{p}}^{2}}}{\mathscr{L}(0)}.\label{Corollary1,8}
\end{equation}
Hence the claim follows.
\end{proof}
\begin{remark}\label{Remark1}
Consider a one-parameter family of Minkowski immersions $\Gamma:\mathbb{S}^{1}\times\left[0,T\right)\rightarrow\mathcal{M}^{2}$ with a general (Minkowski) normal speed $\mathcal{F}$, that is,
\[
\partial_{t}^{N}\Gamma=\mathcal{F}.
\]
Then by following the same procedure in the proof of the Corollary \ref{Corollary1}, the Minkowski length $\mathscr{L}$ can be seen to evolve via the equation
\[
\frac{d}{dt}\mathscr{L}\oo{\Gamma}=-\intmink{\kappa\cdot\mathcal{F}}\geq-\oo{\intmink{\kappa^{2}}\intmink{\mathcal{F}^{2}}}^{\frac{1}{2}},
\] 
with equality if and only if $\mathcal{F}$ is a (positive) scalar multiple of $\kappa$. This establishes \eqref{ACSF} as the steepest descent for the Minkowski length functional.
\end{remark}
\begin{remark}\label{Remark2}
The result $\frac{d}{dt}\mathscr{I}\leq0$ from Corollary \eqref{Corollary1} is critical. This is because by our main theorem we wish to prove convergence to a homothetic rescaling of the isoperimetrix $\tilde{\mathcal{I}}$, which \emph{minimises} the ratio $\mathscr{I}$ amongst all closed curves immersed in $\mathcal{M}^{2}$. 
\end{remark}
\end{section}
\begin{section}{The Minkowski normalised oscillation of curvature}\label{KoscSection}
In this section we study the behaviour of a scale-invariant quantity
\begin{equation}
\kosc=\mathscr{L}\intmink{(\kappa-\overline{\kappa})^{2}},\label{KoscDefinition}
\end{equation}
which we call the \emph{Minkowski} or \emph{anisotropic normalised oscillation
of curvature}. It is the anisotropic analog to an energy that was first 
used by the second author in his study of the curve diffusion flow in the
Euclidean plane \cite{Wheeler5}. The energy is a natural one for our study here
as for the flow \eqref{APH} it is bounded a-priori in $L^{1}$. To see this one
first notes that because $\intmink{(\kappa-\overline{\kappa})}=0$ and
$\intmink{\kappa_{\sigma^{i}}}=0\,(i\geq1)$ because of the periodicity of
$\Gamma$, we may apply the results of Lemma \ref{AppendixLemma1} $p$ times in
succession:
\[
\kosc=\mathscr{L}\intmink{(\kappa-\overline{\kappa})^{2}}\leq\mathscr{L}\oo{\frac{\mathscr{L}}{2\pi}}^{2p}\intmink{\kappa_{\sigma^{p}}^{2}}=-\frac{1}{2(p+1)(2\pi)^{2p}}\frac{d}{dt}(\mathscr{L}^{2(p+1)}).
\]
Therefore
\begin{equation}
\int_{0}^{t}{\kosc(\tau)\,d\tau}\leq\frac{1}{2(p+1)(2\pi)^{2p}}\,\mathscr{L}(\Gamma_{0}).\label{KoscBounded}
\end{equation}
Similarly, by applying Lemma \ref{AppendixLemma2} $p$ times we obtain
\begin{equation}
\int_{0}^{t}{||\kappa-\overline{\kappa}||_{\infty}^{2}\,d\tau}\leq\frac{1}{2p(2\pi)^{2p-1}}\mathscr{L}^{2p}(\gamma_{0}).\label{KoscBoundedLinf}
\end{equation}
We now formulate the evolution equation for $\kosc$.
\begin{prop}\label{Proposition1}
Suppose $\Gamma:\mathbb{S}^{1}\times\left[0,T\right)\rightarrow\mathcal{M}^{2}$ satisfies \eqref{APH}. Define $\mathscr{Q}:=h^{3}\oo{h+h_{\theta\theta}}$, where $h=r^{-1}$ is the radial support function of the indicatrix $\partial\mathcal{U}$. Then
\begin{align}
&\frac{d}{dt}\oo{\kosc+8\mathscr{A}(\tilde{\mathcal{I}})^{2}\ln{\mathscr{L}}}+2\mathscr{L}\intmink{\mathscr{Q}\kappa_{\sigma^{p+1}}^{2}}\nonumber\\
&\quad=\mathscr{L}\intmink{\kappa_{\sigma^{p}}\oo{\kave^{3}+\bar{\kappa}\kave^{2}}_{\sigma^{p}}}-2\mathscr{L}\sum_{l=0}^{p-1}\binom{p}{l}\intmink{\mathscr{Q}_{\sigma^{p-l}}\kappa_{\sigma^{l+1}}\kappa_{\sigma^{p+1}}}.\label{Proposition1,000}
\end{align}
\end{prop}
\begin{proof}
This is a direct calculation:
\begin{align}
\frac{d}{dt}\kosc&=\frac{d}{dt}\oo{\mathscr{L}\intmink{\oo{\kappa-\bar{\kappa}}^{2}}}\nonumber\\
&=-\intmink{\kappa_{\sigma^{p}}^{2}}\cdot\intmink{\oo{\kappa-\bar{\kappa}}^{2}}+2\mathscr{L}\intmink{\oo{\kappa-\bar{\kappa}}\partial_{t}\kappa}\nonumber\\
&\quad+\oo{-1}^{p+1}\intmink{\oo{\kappa-\bar{\kappa}}^{2}\kappa\cdot\kappa_{\sigma^{2p}}}\nonumber\\
&=-\frac{\llll{\kappa_{\sigma^{p}}}_{2}^{2}}{\mathscr{L}}\kosc+2\oo{-1}^{p}\mathscr{L}\intmink{\kave\oo{\oo{h^{3}\oo{h+h_{\theta\theta}}\kappa_{\sigma^{2p+1}}}_{\sigma}+\kappa^{2}\kappa_{\sigma^{2p}}}}\nonumber\\
&\quad+\oo{-1}^{p+1}\mathscr{L}\intmink{\kave^{2}\kappa\cdot\kappa_{\sigma^{2p}}}.\nonumber
\end{align}
Hence
\begin{align}
&\frac{d}{dt}\kosc+\frac{\llll{\kappa_{\sigma^{p}}}_{2}^{2}}{\mathscr{L}}\kosc+2\mathscr{L}\intmink{\oo{\mathcal{Q\kappa_{\sigma}}}_{\sigma^{p}}\kappa_{\sigma^{p+1}}}\nonumber\\
&=\oo{-1}^{p}\mathscr{L}\intmink{\oo{2\oo{\kappa-\bar{\kappa}}\kappa^{2}-\oo{\kappa-\bar{\kappa}}^{2}\kappa}\kappa_{\sigma^{2p}}}\nonumber\\
&=\mathscr{L}\intmink{\kappa_{\sigma^{p}}\oo{\kave^{3}+\bar{\kappa}\kave^{2}}_{\sigma^{p}}}+2\bar{\kappa}^{2}\mathscr{L}\llll{\kappa_{\sigma^{p}}}_{2}^{2}\nonumber\\
&=\mathscr{L}\intmink{\kappa_{\sigma^{p}}\oo{\kave^{3}+\bar{\kappa}\kave^{2}}_{\sigma^{p}}}-8\mathscr{A}(\tilde{\mathcal{I}})^{2}\frac{d}{dt}\ln{\mathscr{L}}.\nonumber
\end{align}
Applying the general Leibniz rule to the last term on the left hand side and rearranging then gives \eqref{Proposition1,000}. 
\end{proof}

\begin{lem}\label{NewLemmaInequality}
If $\kosc\leq1$ then we have
\begin{equation}
\mathscr{L}\intmink{\kappa_{\sigma^{p}}\oo{\kave^{3}+\bar{\kappa}\kave^{2}}_{\sigma^{p}}}\leq \mathscr{L}\oo{c_{1}\,\kosc+c_{2}\,\sqrt{\kosc}}\label{NewLemmaInequality1}
\end{equation}
for some universal constants $c_{i}(p)>0$.
\end{lem}
\begin{proof}
The proof is essentially the same as that of Lemma $7$ from \cite{parkins2015polyharmonic}. Note that the integrand here does not depend on $\mathscr{Q}$ implicitly.
\end{proof}

\begin{cor}
Suppose $\Gamma:\mathbb{S}^{1}\times\left[0,T\right)\rightarrow\mathcal{M}^{2}$ solves \eqref{APH}. Also, define  $\mathscr{Q}_{\star}:=\min\mathscr{Q}>0$. If $\kosc\leq1$, then there exists universal constants $c_{i}(p,\partial\mathcal{U})>0$ such that 
\begin{multline}
\frac{d}{dt}\oo{\kosc+8\mathscr{A}(\tilde{\mathcal{I}})^{2}\ln{\mathscr{L}}}+\frac{\llll{\kappa_{\sigma^{p}}}_{2}^{2}}{\mathscr{L}}\kosc\\
+\mathscr{L}\oo{2\mathscr{Q}_{\star}-c_{1}\kosc-c_{2}\sqrt{\kosc}}\intmink{\kappa_{\sigma^{p+1}}^{2}}\leq0\label{Proposition1,00}
\end{multline}
where $c_{i}=c_{i}\oo{p,\partial\mathcal{U}}$ are universal constants. Therefore, if there exists a positive time $T^{*}$ such that
\begin{equation}
\kosc\oo{t}\leq\min\left\{\frac{c_{2}^{2}+4c_{1}\mathscr{Q}_{\star}-c_{2}\sqrt{c_{2}^{2}+8c_{1}\mathscr{Q}_{\star}}}{4c_{1}^{2}},1\right\}=:2\mathcal{K}^{\star}\,\,\text{for}\,\,t\in\left[0,T^{\star}\right),\label{Proposition1,0}
\end{equation}
then during this time we have the estimate
\begin{equation}
\kosc\oo{t}+\int_{0}^{t}{\frac{\llll{\kappa_{\sigma^{p}}}_{2}^{2}}{\mathscr{L}\oo{\tau}}\kosc\oo{\tau}\,d\tau}\leq\kosc\oo{0}+8\mathscr{A}(\tilde{\mathcal{I}})^{2}\ln\oo{\frac{\mathscr{L}\oo{0}}{\mathscr{L}\oo{t}}}.\label{Proposition1,0000}
\end{equation}
\end{cor}
\begin{proof}
First note that the first term on the right hand side of \eqref{Proposition1,000} was dealt with in Lemma \ref{NewLemmaInequality}. We will therefore only need to pay our attention to the $P-$style terms on the right hand side of \eqref{Proposition1,000}.
We will use the results of Proposition \ref{Proposition1} and Lemma \ref{NewLemmaInequality} but first need a general rule that tells as how to deal with derivatives of $\mathscr{Q}$. This turns out to be quite easy: for any $j\in\mathbb{N}\cup\left\{0\right\}$ we have
\begin{equation}
\mathscr{Q}_{\sigma^{j}}=\sum_{i=1}^{j}\tilde{c}_{i}P_{i}^{j-i}\oo{\kappa},\,\,j\in\mathbb{N},\label{Q1}
\end{equation}
where $\tilde{c}_{i}$ is a function that depends upon the radial support function $h$ and its derivatives $h_{\theta},h_{\theta^{2}},\dots h_{\theta^{i+2}}$. For example, an application of the chain rule gives
$$\mathscr{Q}_{\sigma}=\oo{\frac{\partial}{\partial\theta}\oo{h^{3}\oo{h+h_{\theta\theta}}}\cdot\frac{h}{\oo{h+h_{\theta\theta}}}}\kappa=\tilde{c}_{1}\kappa=\sum_{i=1}^{1}\tilde{c}_{1}P_{i}^{1-i}\oo{\kappa},$$
where $\tilde{c}_{1}\oo{h,h_{\theta},h_{\theta^{2}},h_{\theta^{3}}}=\oo{h^{3}\oo{h+h_{\theta\theta}}}_{\theta}\cdot\frac{h}{\oo{h+h_{\theta\theta}}}$. The proof of this identity comes from a simple inductive argument. Next we combine the identity 
\begin{equation}
P_{k}^{n}\oo{\kappa}=\sum_{l=0}^{k-1}\bar{\kappa}^{l}P_{k-l}^{n}\oo{\kappa-\bar{\kappa}}\label{Pstyle!}
\end{equation}
with \eqref{Q1} to obtain
\begin{equation}
\mathscr{Q}_{\sigma^{j}}=\sum_{i=1}^{j}\tilde{c}_{i}\sum_{m=0}^{i-1}\bar{\kappa}^{m}P_{i-m}^{j-i}\kave,\,\,j\in\mathbb{N},\label{Q2}
\end{equation}
which is a polynomial in $\kave$ with coefficients that depend only on $h$ and its derivatives. In order to use Lemma \ref{AppendixLemma5} correctly to estimate the remaining terms in \eqref{Proposition1,000} we will need to make sure that the highest derivative of $\kave$ is of the order no greater than $(p+1)$. We effectively do this by integrating by parts. If $l=p-1$ we integrate by parts to give
\[
\intmink{\mathscr{Q}_{\sigma}\kappa_{\sigma^{p}}\kappa_{\sigma^{p+1}}}=-\intmink{\mathscr{Q}_{\sigma}\kappa_{\sigma^{p}}\kappa_{\sigma^{p+1}}}-\intmink{\kappa_{\sigma^{p}}^{2}\mathscr{Q}_{\sigma^{2}}}.
\]
Therefore
\begin{align}
&-2\mathscr{L}\intmink{\mathscr{Q}_{\sigma}\kappa_{\sigma^{p}}\kappa_{\sigma^{p+1}}}\nonumber\\
&\quad=\mathscr{L}\intmink{\kappa_{\sigma^{p}}^{2}\mathscr{Q}_{\sigma^{2}}}\nonumber\\
&\quad=\mathscr{L}\sum_{i=1}^{2}\sum_{m=0}^{i-1}\intmink{\kave_{\sigma^{p}}^{2}\tilde{c}_{i}\oo{\bar{\kappa}}^{m}P_{i-m}^{2-i}\kave}\nonumber\\
&\quad\leq c\sum_{i=1}^{2}\sum_{m=0}^{i-1}\mathscr{L}^{1-m}\intmink{\norm{P_{i-m+2}^{2\oo{p+1}-i,p}\kave}}.\label{Proposition1,1}
\end{align}
Note that this expansion does not depend on $l$. Similarly, if $l<p-1$ we have
\begin{align}
&-2\mathscr{L}\binom{p}{l}\intmink{\mathscr{Q}_{\sigma^{p-l}}\kappa_{\sigma^{l+1}}\kappa_{\sigma^{p+1}}}\nonumber\\
&=2\mathscr{L}\binom{p}{l}\intmink{\kave_{p}\oo{\mathscr{Q}_{\sigma^{p-l+1}}\kave_{\sigma^{l+1}}+\mathscr{Q}_{\sigma^{p-l}}\kave_{\sigma^{l+2}}}}\nonumber\\
&\leq c\sum_{i=1}^{p-l}\sum_{m=0}^{i-1}\mathscr{L}^{1-m}\intmink{\norm{P_{i-m+2}^{2\oo{p+1}-i,p}\kave}}.\label{Proposition1,2}
\end{align}
Note that the right hand side of \eqref{Proposition1,2} takes the same form of the right hand side of \eqref{Proposition1,1} with $l=p-2$. 
Next for $m\in\left\{0,1,\dots,i-1\right\}$, Lemma \ref{AppendixLemma5} gives us 
\begin{align}
&\mathscr{L}^{1-m}\intmink{\norm{P_{i-m+2}^{2\oo{p+1}-i,p}\kave}}\nonumber\\
&\leq c\mathscr{L}^{1-m}\cdot\mathscr{L}^{1-\oo{i-m+2}-2\oo{p+1}+i}\oo{\kosc}^{\frac{i-m}{2}+\frac{i+m}{4\oo{p+1}}}\oo{\mathscr{L}^{2p+3}\intmink{\kappa_{\sigma^{p+1}}^{2}}}^{1-\frac{i+m}{4\oo{p+1}}}\nonumber\\
&\leq c\mathscr{L}^{-2\oo{p+1}}\oo{\kosc}^{\frac{i-m}{2}}\oo{\mathscr{L}^{2p+3}\intmink{\kappa_{\sigma^{2p+3}}^{2}}}^{\frac{i+m}{4\oo{p+1}}}\oo{\mathscr{L}^{2p+3}\intmink{\kappa_{\sigma^{p+1}}^{2}}}^{1-\frac{i+m}{4\oo{p+1}}}\nonumber\\
&\leq c\mathscr{L}\oo{\kosc}^{\frac{i-m}{2}}\intmink{\kappa_{\sigma^{p+1}}^{2}}\\
&\leq c\sqrt{\kosc}\mathscr{L}\intmink{\kappa_{\sigma^{p+1}}^{2}}.\label{Proposition1,2*}
\end{align}
This last step follows because by assumption $\oo{i-m}/2\geq1/2$, and $\kosc\leq1$. 
Combining this and \eqref{NewLemmaInequality1}, and substituting back into \eqref{Proposition1,000} then gives \eqref{Proposition1,00}.
Integrating this result then gives equation \eqref{Proposition1,000} under the small-energy assumption \eqref{Proposition1,00}.
\end{proof}

\begin{prop}\label{Proposition2}
Suppose that $\Gamma:\mathbb{S}^{1}\times\left[0,T\right)\rightarrow\mathcal{M}^{2}$ solves \eqref{APH}. Additionally, suppose that $\Gamma_{0}$ is a simple closed curve satisfying 
\begin{equation}
\kosc\oo{\Gamma_{0}}\leq\mathcal{K}^{\star}\,\,\text{and}\,\,\mathscr{I}\oo{\Gamma_{0}}\leq\exp\oo{\mathcal{K}^{\star}/8\mathscr{A}(\tilde{\mathcal{I}})^{2}}.\label{Proposition2,1}
\end{equation}
Then we have
\begin{equation}
\kosc\oo{\Gamma}\leq2\mathcal{K}^{\star}\,\,\text{for}\,\,t\in\left[0,T\right).\label{Proposition2,2}
\end{equation}
\end{prop}
\begin{proof}
Assume for the sake of contradiction that $\kosc$ does \emph{not} remain bounded by $2\mathcal{K}^{\star}$. Then we can find a maximal $T^{\star}<T$ such that
\[
\kosc\oo{\Gamma}\leq2\mathcal{K}^{\star}\,\,\text{for}\,\,t\in\left[0,T^{\star}\right).
\]
Hence by \eqref{Proposition1,000} the following identity holds for $t\in\left[0,T^{\star}\right)$:
\begin{equation}
\kosc\oo{\Gamma_{t}}\leq\kosc\oo{\Gamma_{0}}+8\mathscr{A}(\tilde{\mathcal{I}})^{2}\ln\oo{\mathscr{L}\oo{\Gamma_{0}}/\mathscr{L}\oo{\Gamma_{t}}}.\label{Proposition2,3}
\end{equation}
Next the anisotropic isoperimetric inequality gives
\[
\frac{\mathscr{L}\oo{\Gamma_{0}}}{\mathscr{L}\oo{\Gamma_{t}}}\leq\frac{\mathscr{L}\oo{\Gamma_{0}}}{\sqrt{4\mathscr{A}\oo{\Gamma_{0}}\mathscr{A}(\tilde{\mathcal{I}})}}=\sqrt{\mathscr{I}\oo{\Gamma_{0}}},
\] 
and so \eqref{Proposition2,3} implies that
\begin{equation}
\kosc\oo{\Gamma_{t}}\leq\kosc\oo{\Gamma_{0}}+4\mathscr{A}(\tilde{\mathcal{I}})^{2}\ln\mathscr{I}\oo{\Gamma_{0}}\leq\mathcal{K}^{\star}+\mathcal{K}^{\star}/2=3\mathcal{K}^{\star}/2,\label{Proposition2,4}
\end{equation}
which is \emph{strictly} less than $2\mathcal{K}^{\star}$. The penultimate step here follows from the assumptions \eqref{Proposition2,1}. Taking $t\nearrow T^{\star}$ in \eqref{Proposition2,4} contradicts the definition of $T^{\star}$, and so we conclude that such a maximal $T^{\star}$ can not exist. The claim then follows.
\end{proof}

\begin{lem}\label{Lemma3}
Suppose that $\Gamma:\mathbb{S}^{1}\times\left[0,T\right)\rightarrow\mathcal{M}^{2}$ solves \eqref{APH}. Then for any $m\geq0$ we have
\begin{equation}
\frac{d}{dt}\intmink{\kappa_{\sigma^{m}}^{2}}\leq c\oo{m,p}\oo{\intmink{\kappa^{2}}}^{2\oo{m+p}+3}\label{Lemma3,0}
\end{equation}
for some constant $c>0$. Here $\kappa_{\sigma^{m}}$ refers to the $m^{th}$ repeated derivative of $\kappa$ with respect to $\sigma$.
\end{lem}
\begin{proof}
We first work out the commutator operator $\cc{\partial_{t},\partial_{\sigma}}$. This is relatively straightforward:
\begin{align}
\partial_{t\sigma}&=\partial_{t}\oo{r\norm{\Gamma_{u}}^{-1}\partial_{u}}\nonumber\\
&=\oo{\oo{-1}^{p+1}h_{\theta}\kappa_{\sigma^{2p+1}}\norm{\Gamma_{u}}^{-1}+r\oo{-1}^{p}\oo{h\cdot h_{\theta}\kappa_{\sigma^{2p+1}}+\kappa\cdot\kappa_{\sigma^{2p}}}\norm{\Gamma_{u}}^{-1}}\partial_{u}+\partial_{\sigma t}\nonumber\\
&=\oo{-1}^{p}\kappa\cdot\kappa_{\sigma^{2p}}\,\partial_{\sigma}+\partial_{\sigma t}.\nonumber
\end{align}
Hence
\begin{equation}
\cc{\partial_{t},\partial_{\sigma}}=\oo{-1}^{p}\kappa\cdot\kappa_{\sigma^{2p}}\,\partial_{\sigma}.\label{Lemma3,1}
\end{equation}
Repeated applications of \eqref{Lemma3,1} then gives, for any $m\in\mathbb{N}_{0}$:
\begin{align}
\partial_{t}\kappa_{\sigma^{m}}&=\partial_{\sigma}^{m}\partial_{t}\kappa+\oo{-1}^{p}\sum_{i=0}^{m-1}\partial_{\sigma}^{i}\oo{\kappa\cdot\kappa_{\sigma^{2p}}\cdot\kappa_{\sigma^{m-i}}}\nonumber\\
&=\oo{-1}^{p}\partial_{\sigma}^{m+1}\oo{\mathscr{Q}\kappa_{\sigma^{2p+1}}}+\oo{-1}^{p}\sum_{i=0}^{m}\partial_{\sigma}^{i}\oo{\kappa\cdot\kappa_{\sigma^{2p}}\cdot\kappa_{\sigma^{m-i}}}.\nonumber
\end{align}
Hence
\begin{align}
&\frac{d}{dt}\intmink{\kappa_{\sigma^{m}}^{2}}\nonumber\\
&=2\oo{-1}^{p}\intmink{\kappa_{\sigma^{m}}\partial_{\sigma}^{m+1}\oo{\mathscr{Q}\kappa_{\sigma^{2p+1}}}}+2\oo{-1}^{p}\sum_{i=0}^{m}\intmink{\kappa_{\sigma^{m}}\partial_{\sigma}^{i}\oo{\kappa\cdot\kappa_{\sigma^{2p}}\cdot\kappa_{\sigma^{m-i}}}}\nonumber\\
&\quad+\oo{-1}^{p+1}\intmink{\kappa_{\sigma^{m}}^{2}\cdot\kappa\cdot\kappa_{\sigma^{2p}}}\nonumber\\
&=2\oo{-1}^{p}\intmink{\kappa_{\sigma^{m}}\partial_{\sigma}^{m+1}\oo{\mathscr{Q}\kappa_{\sigma^{2p+1}}}}+2\oo{-1}^{p}\sum_{i=1}^{m}\intmink{\kappa_{\sigma^{m}}\partial_{\sigma}^{i}\oo{\kappa\cdot\kappa_{\sigma^{2p}}\cdot\kappa_{\sigma^{m-i}}}}\nonumber\\
&\quad+\oo{-1}^{p}\intmink{\kappa_{\sigma^{m}}^{2}\cdot\kappa\cdot\kappa_{\sigma^{2p}}}.\label{Lemma3,2}
\end{align}
We have to consider the cases $p>m$ and $p\leq m$ separately in order to use Lemma \ref{AppendixLemma5} correctly. The idea is that we do not want any derivatives of $\kappa$ that are higher than order $m+p$.

If $p>m$, we write $p=m+\alpha$ for some $\alpha\in\mathbb{N}$. The equation \eqref{Lemma3,2} then becomes
\begin{align}
&\frac{d}{dt}\intmink{\kappa_{\sigma^{m}}^{2}}\nonumber\\
&=2\oo{-1}^{p+m+1}\intmink{\mathscr{Q}\kappa_{\sigma^{2m+1}}\cdot\kappa_{\sigma^{m+p+1+\alpha}}}\nonumber\\
&\quad+2\sum_{i=1}^{m}\oo{-1}^{m+i}\intmink{\kappa\cdot\kappa_{\sigma^{m-i}}\cdot\kappa_{\sigma^{m+i}}\cdot\kappa_{\sigma^{m+p+\alpha}}}+\oo{-1}^{p}\intmink{\kappa\cdot\kappa_{\sigma^{m}}^{2}\cdot\kappa_{\sigma^{m+p+\alpha}}}\nonumber\\
&=-2\intmink{\kappa_{\sigma^{m+p+1}}\partial_{\sigma}^{\alpha}\oo{\mathscr{Q}\kappa_{\sigma^{2m+1}}}}+2\sum_{i=1}^{m}\oo{-1}^{m+i+\alpha}\intmink{\kappa_{\sigma^{m+p}}\partial_{\sigma}^{\alpha}\oo{\kappa\cdot\kappa_{\sigma^{m-i}}\cdot\kappa_{\sigma^{m+i}}}}\nonumber\\
&\quad+\oo{-1}^{p+\alpha}\intmink{\kappa_{\sigma^{m+p}}\partial_{\sigma}^{\alpha}\oo{\kappa\cdot\kappa_{\sigma^{m}}^{2}}}\nonumber\\
&=-2\intmink{\mathscr{Q}\kappa_{\sigma^{m+p+1}}^{2}}-2\sum_{j=1}^{\alpha}\binom{\alpha}{j}\intmink{\mathscr{Q}_{\sigma^{j}}\kappa_{\sigma^{m+p+1}}\cdot\kappa_{\sigma^{2m+1+\alpha-j}}}\nonumber\\
&\quad+2\sum_{i=1}^{m}\oo{-1}^{m+i+\alpha}\intmink{\kappa_{\sigma^{m+p}}\partial_{\sigma}^{\alpha}\oo{\kappa\cdot\kappa_{\sigma^{m-i}}\cdot\kappa_{\sigma^{m+i}}}}\\
&\quad+\oo{-1}^{p+\alpha}\intmink{\kappa_{\sigma^{m+p}}\partial_{\sigma}^{\alpha}\oo{\kappa\cdot\kappa_{\sigma^{m}}^{2}}}\nonumber\\
&\leq-2\intmink{\mathscr{Q}\kappa_{\sigma^{m+p+1}}^{2}}-2\sum_{j=1}^{\alpha}\binom{\alpha}{j}\intmink{\mathscr{Q}_{\sigma^{j}}\kappa_{\sigma^{m+p+1}}\cdot\kappa_{\sigma^{m+p+1-j}}}\nonumber\\
&\quad+c\intmink{\norm{P_{4}^{2\oo{m+p},m+p}\oo{\kappa}}}.\label{Lemma3,3}
\end{align}
In a similar manner to \eqref{Proposition1,1}, we treat the cases $j=1,j<1$ separately in the sum on the right hand side of \eqref{Lemma3,3}. For $j=1$ we have
\[
\intmink{\mathscr{Q}_{\sigma}\kappa_{\sigma^{m+p+1}}\cdot\kappa_{\sigma^{m+p}}}=-\intmink{\mathscr{Q}_{\sigma}\kappa_{\sigma^{m+p+1}}\cdot\kappa_{\sigma^{m+p}}}-\intmink{\mathscr{Q}_{\sigma^{2}}\kappa_{\sigma^{m+p}}^{2}},
\]
where we have used integration by parts. Hence
\begin{align}
\intmink{\mathscr{Q}_{\sigma}\kappa_{\sigma^{m+p+1}}\cdot\kappa_{\sigma^{m+p}}}&=-\frac{1}{2}\intmink{\mathscr{Q}_{\sigma^{2}}\kappa_{\sigma^{m+p}}^{2}}\\
&=\sum_{i=1}^{2}\intmink{\tilde{c}_{i}P_{i}^{2-i}\oo{\kappa}\kappa_{\sigma^{m+p}}^{2}}\nonumber\\
&\leq c\sum_{i=1}^{2}\intmink{\norm{P_{i+2}^{2\oo{m+p+1}-i,m+p}\oo{\kappa}}}.\label{Lemma3,4}
\end{align}
For $j>1$ we can integrate by parts once, yielding
\begin{align}
&\intmink{\mathscr{Q}_{\sigma^{j}}\kappa_{\sigma^{m+p+1}}\cdot\kappa_{\sigma^{m+p+1-j}}}\nonumber\\
&=-\intmink{\kappa_{\sigma^{m+p}}\oo{\mathscr{Q}_{\sigma^{j+1}}\kappa_{\sigma^{m+p+1-j}}+\mathscr{Q}_{\sigma^{j}}\kappa_{\sigma^{m+p+2-j}}}}\nonumber\\
&=\sum_{i=1}^{j+1}\intmink{\tilde{c}_{i}\kappa_{\sigma^{m+p}}\kappa_{\sigma^{m+p+1-j}}P_{i}^{j+1-i}\oo{\kappa}}\nonumber\\
&\quad+\sum_{i=1}^{j}\intmink{\tilde{c}_{i}\kappa_{\sigma^{m+p}}\kappa_{\sigma^{m+p+2-j}}P_{i}^{j-i}\oo{\kappa}}\nonumber\\
&\leq c\sum_{i=1}^{j+1}\intmink{\norm{P_{i+2}^{2\oo{m+p+1}-i,m+p}\oo{\kappa}}}.\label{Lemma3,4,2}
\end{align}
Substituting \eqref{Lemma3,4} and \eqref{Lemma3,4,2} into \eqref{Lemma3,3} then yields
\begin{align}
&\frac{d}{dt}\intmink{\kappa_{\sigma^{m}}^{2}}+2\mathscr{Q}_{\star}\intmink{\kappa_{\sigma^{m+p+1}}^{2}}\nonumber\\
&\leq c\sum_{j=1}^{\alpha}\sum_{i=1}^{j+1}\intmink{\norm{P_{i+2}^{2\oo{m+p+1}-i,m+p}\oo{\kappa}}}\nonumber\\
&\leq \varepsilon\intmink{\kappa_{\sigma^{m+p+1}}^{2}}+c\oo{\varepsilon,m,p}\oo{\intmink{\kappa^{2}}}^{2\oo{m+p}+3}.\label{Lemma3,5}
\end{align}
Here we have used \eqref{AppendixLemma5,2} from Lemma \ref{AppendixLemma5} in the last step.

Next, if $p\leq m$ then we write $m=p+\alpha,\,\alpha\in\mathbb{N}\cup\left\{0\right\}$ and equation \eqref{Lemma3,2} becomes
\begin{align}
&\frac{d}{dt}\intmink{\kappa_{\sigma^{m}}^{2}}\nonumber\\
&=2\oo{-1}^{m+p+1}\intmink{\mathscr{Q}\kappa_{\sigma^{2p+1}}\kappa_{\sigma^{2m+1}}}+2\sum_{i=0}^{p}\oo{-1}^{p+i}\intmink{\kappa\cdot\kappa_{\sigma^{m-i}}\kappa_{\sigma^{m+i}}\kappa_{\sigma^{2p}}}\nonumber\\
&\quad+2\sum_{i=p+1}^{p+\alpha}\oo{-1}^{p+i}\intmink{\kappa\cdot\kappa_{\sigma^{m-i}}\kappa_{\sigma^{m+i}}\kappa_{\sigma^{2p}}}+\oo{-1}^{p+1}\intmink{\kappa_{\sigma^{m}}^{2}\cdot\kappa\cdot\kappa_{\sigma^{2p}}}\nonumber\\
&\leq2\oo{-1}^{m+p+1}\intmink{\mathscr{Q}\kappa_{\sigma^{2p+1}}\kappa_{\sigma^{m+p+1+\alpha}}}\nonumber\\
&\quad+2\sum_{i=p+1}^{p+\alpha}\oo{-1}^{p+i}\intmink{\kappa\cdot\kappa_{\sigma^{m-i}}\kappa_{\sigma^{m+p+\oo{i-p}}}\kappa_{\sigma^{2p}}}+\intmink{\norm{P_{4}^{2\oo{m+p},m+p}\oo{\kappa}}}\nonumber\\
&=-2\intmink{\kappa_{\sigma^{m+p+1}}\partial_{\sigma}^{\alpha}\oo{\mathscr{Q}\kappa_{\sigma^{2p+1}}}}+2\sum_{i=p+1}^{p+\alpha}\intmink{\kappa_{\sigma^{m+p}}\partial_{\sigma}^{i-p}\oo{\kappa\cdot\kappa_{\sigma^{m-i}}\cdot\kappa_{\sigma^{2p}}}}\nonumber\\
&\quad+\intmink{\norm{P_{4}^{2\oo{m+p},m+p}\oo{\kappa}}}\nonumber\\
&\leq-2\sum_{i=0}^{\alpha}\binom{\alpha}{i}\intmink{\kappa_{\sigma^{m+p+1}}\mathscr{Q}_{\sigma^{i}}\kappa_{\sigma^{2p+1+\alpha-i}}}+\intmink{\norm{P_{4}^{2\oo{m+p},m+p}\oo{\kappa}}}.\nonumber
\end{align}
Here we have used the fact that for $i\leq p$, $\max\left\{m-i,m+i,2p\right\}\leq m+p$ in the second and last steps. Hence
\begin{align}
&\frac{d}{dt}\intmink{\kappa_{\sigma^{m}}^{2}}+2\mathscr{Q}_{\star}\intmink{\kappa_{\sigma^{m+p+1}}^{2}}\nonumber\\
&\leq-2\sum_{j=1}^{\alpha}\binom{\alpha}{j}\intmink{\kappa_{\sigma^{m+p+1}}\mathscr{Q}_{\sigma^{j}}\kappa_{\sigma^{m+p+1-j}}}+\intmink{\norm{P_{4}^{2\oo{m+p},m+p}\oo{\kappa}}}.\label{Lemma3,6}
\end{align}
Now, we have already dealt with terms in the summation on the right back in \eqref{Lemma3,4} and \eqref{Lemma3,4,2}. Hence \eqref{Lemma3,6} becomes
\begin{align}
&\frac{d}{dt}\intmink{\kappa_{\sigma^{m}}^{2}}+2\mathscr{Q}_{\star}\intmink{\kappa_{\sigma^{m+p+1}}^{2}}\nonumber\\
&\leq c\sum_{j=1}^{\alpha}\sum_{i=1}^{j+1}\intmink{\norm{P_{i+2}^{2\oo{m+p+1}-i,m+p}\oo{\kappa}}}\nonumber\\
&\leq \varepsilon\intmink{\kappa_{\sigma^{m+p+1}}^{2}}+c\oo{\varepsilon,m,p}\oo{\intmink{\kappa^{2}}}^{2\oo{m+p}+3}.\nonumber
\end{align}
Here we have used \eqref{AppendixLemma5,2} from Lemma \ref{AppendixLemma5} in the last step.
Comparing this to \eqref{Lemma3,5} we see that the evolution equation for $\intmink{\kappa_{\sigma^{m}}^{2}}$ can be estimated in the same way, regardless of the sign of $m-p$. We conclude that
\[
\frac{d}{dt}\intmink{\kappa_{\sigma^{m}}^{2}}+\oo{2\mathscr{Q}_{\star}-\varepsilon}\intmink{\kappa_{\sigma^{m+p+1}}^{2}}\leq c\oo{\varepsilon,m,p}\oo{\intmink{\kappa^{2}}}^{2\oo{m+p}+3}.
\]
Choosing $\varepsilon>0$ sufficiently small then gives \eqref{Lemma3,0}. 
\end{proof}

The preceeding lemma allows us to characterise the finite-time singularities for the flows \eqref{APH}. In the same vein as earlier work from Kuwert, Sch\"{a}tzle and Dziuk's paper on the evolution of elastic curves in $\mathbb{R}^{n}$ \cite{Kuwert3}, we show that if our flow becomes extinct in finite time, then we must encounter an $L^{2}$ curvature concentration as we approach the maximal time $T$.
\begin{lem}\label{Lemma5}
Suppose $\Gamma:\mathbb{S}^{1}\times\left[0,T\right)\rightarrow\mathcal{M}^{2}$ is a maximal solution to \eqref{APH}. If $T<\infty$, then
\begin{equation}
\intmink{\kappa^{2}}\Big|_{t}\geq c\oo{T-t}^{-1/2\oo{p+1}}\label{Lemma5,0}
\end{equation} 
for some universal constant $c>0$.
\end{lem}

\begin{proof}
By Lemma \ref{Lemma3}, it will be enough to prove that $\limsup_{t\nearrow T}\intmink{\kappa^{2}}=\infty$. We assume for the sake of contradiction that $\intmink{\kappa^{2}}\leq\varrho<\infty$ for all $t<T$. Our aim is to show that this assumption implies that our one-parameter family of solutions is smooth right up to the maximal time of existence $T$. Local existence results will then allow us to extend the flow smoothly past time $T$ to an interval $\left[0,T+\delta\right)$, contradicting the maximality of $T$. As mentioned above, the argument is essentially an extension of a result from \cite{Kuwert3}, but has been included for the convenience of the reader.
\\\\
Now, by \eqref{Lemma3,0} we have $\frac{d}{dt}\intmink{\kappa_{{\sigma}^{m}}^{2}}\leq c\cdot\varrho^{2\oo{m+p}+3}$ for any $m\in\mathbb{N}_{0}$. Hence
\[
\intmink{\kappa_{{\sigma}^{m}}^{2}}\leq\intmink{\kappa_{{\sigma}^{m}}^{2}}\Big|_{t=0}+c\cdot\varrho^{2\oo{m+p}+3}T\leq c_{m}\oo{\Gamma_{0},\varrho,T}.
\]
Combining this with Lemma \ref{AppendixLemma2}, we have for $m\geq1$:
\begin{equation}
\llll{\kappa_{\sigma^{m}}}_{\infty}^{2}\leq\frac{\mathscr{L}}{2\pi}\intmink{\kappa_{{\sigma}^{m+1}}^{2}}\leq\frac{\mathscr{L}\oo{\Gamma_{0}}}{2\pi}\cdot c_{m+1}\oo{\Gamma_{0},\varrho,T}\leq d_{m}\oo{\Gamma_{0},\varrho,T},\label{Lemma5,1}
\end{equation}
where $d_{m}$ is a new universal constant. Since we intending to show that $\llll{\partial_{u}^{m}\Gamma}_{\infty}<c_{m}$ for every $m\in\mathbb{N}$, we start by showing that
\begin{equation}
\llll{\partial_{\sigma}^{m}\Gamma}_{\infty}<\tilde{c}_{m}\oo{\Gamma_{0},\varrho,T}\,\,\text{for every}\,\,m\in\mathbb{N},\label{Lemma5,2}
\end{equation}  
and proceed from there. Assume that \eqref{Lemma5,2} is true for $m=1,2,\dots,k$. Then for $m=k+1$, multiple applications of the commutator formula \eqref{Lemma3,1} gives
\begin{align}
\partial_{t}\partial_{\sigma}^{k+1}\Gamma&=\partial_{\sigma}^{k+1}\partial_{t}\Gamma-\sum_{i=0}^{k}\partial_{\sigma}^{i}\oo{\kappa\cdot\kappa_{\sigma\sigma}\partial_{\sigma}^{k+1-i}\Gamma}\nonumber\\
&=-\partial_{\sigma}^{k+1}\oo{\kappa_{\sigma\sigma}N}-\sum_{i=0}^{k}\partial_{\sigma}^{i}\oo{\kappa\cdot\kappa_{\sigma\sigma}\partial_{\sigma}^{k+1-i}\Gamma}\nonumber.
\end{align}
This implies that
\begin{equation}
\norm{\partial_{t}\partial_{\sigma}^{k+1}\Gamma}\leq c\oo{1+\norm{\partial_{\sigma}^{k+1}\Gamma}}.\label{Lemma5,4}
\end{equation}
The constant $c$ here depends only on the $d_{m}$ constants from \eqref{Lemma5,1}, along with the $\tilde{c}_{m},m=1,2\dots k$ from \eqref{Lemma5,2} and $\partial\mathcal{U}$. We have also used the identity
\[
\llll{\partial_{\sigma}^{m}N}_{\infty}\leq c\oo{h,h_{\theta},\dots,h_{\theta^{m+1}},\kappa,\kappa_{\sigma},\dots,\kappa_{\sigma^{m+1}}}\leq c\oo{\Gamma_{0},\varrho,T,\partial\mathcal{U}},
\]
which can be attained inductively by combining the Frenet equations \eqref{MinkowskiFrenet2} and identity \eqref{Q1}.  Using the Kato inequality on \eqref{Lemma5,4} gives
\begin{equation}
\norm{\partial_{t}\norm{\partial_{\sigma}^{k+1}\Gamma}}\leq c\oo{1+\norm{\partial_{\sigma}^{k+1}\Gamma}}.\nonumber
\end{equation}
Our previous inequality gives 
\[
d\ln\oo{1+\norm{\partial_{\sigma}^{k+1}\Gamma}}\leq c\,dt,
\]
which implies that for any $t\in\left[0,T\right)$ we have the following estimate:
\[
\llll{\partial_{\sigma}^{k+1}\Gamma}\Big|_{t}\leq \oo{1+\llll{\partial_{\sigma}^{k+1}\Gamma}\Big|_{0}}e^{ct}-1<\oo{1+\llll{\partial_{\sigma}^{k+1}\Gamma}\Big|_{0}}e^{cT}-1<\tilde{c}_{k+1}\oo{\varrho,\Gamma_{0},T}.
\]
This completes the inductive step and proves \eqref{Lemma5,2}. To jump from \eqref{Lemma5,2} to proving that
\begin{equation}
\llll{\partial_{u}^{m}\Gamma}_{\infty}<c_{m}\oo{\varrho,\Gamma_{0},T,\partial\mathcal{U}}\,\,\text{for every}\,\,m\in\mathbb{N},\label{Lemma5,5}
\end{equation}
we first employ the identities $\partial_{s}=\frac{\kappa}{h+h_{\theta\theta}}\partial_{\theta}$ and $\partial_{s}=h\partial_{\sigma}$, which combined with \eqref{Lemma5,1} and \eqref{Lemma5,2} gives
\begin{equation}
\llll{\partial_{s}^{m}\Gamma}_{\infty}\leq c\oo{h,\dots, h_{\theta^{m}},\kappa,\dots,\kappa_{\sigma^{m-1}}}\sum_{i=1}^{m}\llll{\partial_{\sigma}^{m}\Gamma}_{\infty}<\bar{c}_{m}\oo{\varrho,\Gamma_{0},T,\partial\mathcal{U}}\label{Lemma5,6}
\end{equation}
for every $m\in\mathbb{N}$. Here $\bar{c}_{m}$ is a new universal constant. Applying the same reasoning to \eqref{Lemma5,1} above also gives
\begin{equation}
\llll{\kappa_{s^{m}}}_{\infty}\leq c\oo{\Gamma_{0},\varrho,T,\partial\mathcal{U}}.\label{Lemma5,7}
\end{equation}

Next by following the inductive argument used in the proof of Theorem 3.1 from \cite{Kuwert3}, we have 
\[
\llll{\partial_{u}^{m}\norm{\Gamma_{u}}}_{\infty}\leq c\oo{\varrho,\Gamma_{0},T,\partial\mathcal{U}}\,\,\text{for every}\,\,m\in\mathbb{N}.
\] 
Unfortunately, applying the Kato inequality on this result does not give \eqref{Lemma5,5} because the inequality is the opposite to what we need. Luckily, following along in the inductive argument of \cite{Kuwert3} gives
\begin{equation}
\llll{\partial_{u}^{m}\Gamma}_{\infty}\leq c\llll{\Gamma_{u}}_{\infty}\llll{\partial_{s}^{m}\Gamma}_{\infty}+\llll{P}_{\infty}.\nonumber
\end{equation}
where $P=P\oo{\norm{\Gamma_{u}},\partial_{u}\norm{\Gamma_{u}},\dots,\partial_{u}^{m-1}\norm{\Gamma_{u}},\kappa,\kappa_{s},\dots,\kappa_{s^{m+1}},\partial\mathcal{U}}$ is a polynomial. Applying \eqref{Lemma5,6} and \eqref{Lemma5,7} to this inequality then proves identity \eqref{Lemma5,5}. Therefore $\Gamma\oo{\cdot,t}$ is smooth right up until time $T$, and hence can be extended to some larger time interval $\left[0,T+\delta\right)$. This contradicts the maximality of $T$. Hence the assumption that $\limsup_{t\nearrow T}\intmink{\kappa^{2}}<\infty$ must have been incorrect, and we conclude that 
\[
\limsup_{t\nearrow T}\intmink{\kappa^{2}}=\infty. 
\]
The result then follows immediately from Lemma \ref{Lemma3} with $m=0$. 
\end{proof}

\section{Global analysis of the flow}

The characterisation of finite-time singularities for the flow \eqref{APH}
above allows us to prove the first part of our main theorem: long-time
existence ($T=\infty$).
\begin{thm}\label{Theorem1}
Suppose that $\Gamma:\mathbb{S}^{1}\times\left[0,T\right)\rightarrow\mathcal{M}^{2}$ solves \eqref{APH}. Additionally, suppose that $\Gamma_{0}$ is a simple closed curve satisfying
\begin{equation}
\kosc\oo{\Gamma_{0}}\leq\mathcal{K}^{\star}\,\,\text{and}\,\,\mathscr{I}\oo{\Gamma_{0}}\leq\exp\oo{\mathcal{K}^{\star}/8\mathscr{A}(\tilde{\mathcal{I}})^{2}}.\nonumber
\end{equation}
Then $T=\infty$.
\end{thm}
\begin{proof}
Suppose for the sake of contradiction that $T<\infty$. Then by Lemma \ref{Lemma5} we have
\[
\intmink{\kappa^{2}}\nearrow\infty\,\,\text{as}\,\,t\nearrow T.
\]
This implies that $\kosc$ must diverge as we approach time $T$ as well, since
\[
\kosc\oo{\Gamma}=\mathscr{L}\intmink{\kappa^{2}}-4\mathscr{A}(\tilde{\mathcal{I}})^{2}\geq\sqrt{4\mathscr{A}\oo{\Gamma_{0}}\,\mathscr{A}(\tilde{\mathcal{I}})}\intmink{\kappa^{2}}-4\mathscr{A}(\tilde{\mathcal{I}})^{2}.
\]
Here the last step follows from the isoperimetric inequality. This result directly contradicts Proposition \ref{Proposition2} where we showed that $\kosc$ remained bounded. We conclude that the assumption that $T<\infty$ must have been false. The result then follows.
\end{proof}
Recall we know that if $\Gamma:\mathbb{S}^{1}\times\left[0,T\right)\rightarrow\mathcal{M}^{2}$ is a $2\oo{p+1}$-anisotropic polyharmonic curve flow and satisfies the hypothesis of Theorem \ref{Theorem1}, then $T=\infty$. Hence identity \eqref{KoscBounded} then implies that under these same assumptions we have
\begin{equation}
\kosc\in L^{1}\oo{\left[0,\infty\right)},\,\,\text{with}\,\,\int_{0}^{\infty}{\kosc(\tau)\,d\tau}\leq \frac{1}{2(p+1)(2\pi)^{2p}}\mathscr{L}^{2(p+1)}\oo{\Gamma_{0}}<\infty.\label{KoscL1}
\end{equation}
So we can conclude that the quantity $\kosc$ is approaching zero at an $\varepsilon$-dense set for sufficiently large times. 
However, at the present time we have not ruled out the possibility that $\kosc$ gets smaller and smaller as $t$ gets large, whilst vibrating with higher and higher frequency, remaining in $L^{1}\oo{\left[0,\infty\right)}$ whilst never actually fully dissipating to zero. To rule out this from happening, it is enough to show that $\norm{\frac{d}{dt}\kosc}$ remains bounded by a universal constant for all time. To do so we will need to first show that $\llll{\kappa_{\sigma^{p}}}_{2}^{2}$ remains bounded. We will address this issue with the following proposition.

\begin{prop}\label{LongTimeProp1}
Suppose $\gamma:\mathbb{S}^{1}\times\left[0,T\right)\rightarrow\mathbb{R}^{2}$ solves \eqref{APH} and that $\Gamma_{0}$ has positive enclosed area. There exists a $\varepsilon_{0}>0$ (with $\varepsilon_{0}\leq\mathcal{K}^{\star}$) such that if
\[
\kosc\oo{\Gamma_{0}}<\varepsilon_{0}\,\,\text{and}\,\,I\oo{0}<\exp\oo{\varepsilon_{0}/8\mathscr{A}(\tilde{\mathcal{I}})^{2}},
\]
then $\llll{k_{s^{p}}}_{2}^{2}$ remains bounded for all time. In particular
\[
\intmink{\kappa_{\sigma^{p}}^{2}}\leq c\oo{\Gamma_{0}},
\]
and so $\llll{k_{s^{p}}}_{2}^{2}$ can be controlled a priori.
\end{prop}
\begin{proof}
Applying \eqref{Lemma3,2} with $m=p$ gives
\begin{align}
&\frac{d}{dt}\intmink{\kappa_{\sigma^{p}}^{2}}+2\mathscr{Q}_{\star}\intmink{\kappa_{\sigma^{2p+1}}^{2}}\nonumber\\
&\leq 2\oo{-1}^{p}\sum_{i=1}^{p}\intmink{\kappa_{\sigma^{p}}\partial_{\sigma}^{i}\oo{\kappa\cdot\kappa_{\sigma^{p-i}}\cdot\kappa_{\sigma^{2p}}}}+\oo{-1}^{p}\intmink{\kappa_{\sigma^{p}}^{2}\cdot\kappa\cdot\kappa_{\sigma^{2p}}}\nonumber\\
&=2\sum_{i=1}^{p}\oo{-1}^{p+i}\intmink{\kappa\cdot\kappa_{\sigma^{p-i}}\cdot\kappa_{\sigma^{p+i}}\cdot\kappa_{\sigma^{2p}}}+\oo{-1}^{p}\intmink{\kappa_{\sigma^{p}}^{2}\cdot\kappa\cdot\kappa_{\sigma^{2p}}}\nonumber\\
&=2\sum_{i=1}^{p}\oo{-1}^{p+i}\intmink{\oo{\kappa-\bar{\kappa}+\bar{\kappa}}\kave_{\sigma^{p-i}}\kave_{\sigma^{p+i}}\kave_{\sigma^{2p}}}\nonumber\\
&+2\intmink{\kappa^{2}\kave_{\sigma^{2p}}^{2}}+\oo{-1}^{p}\intmink{\kappa_{\sigma^{p}}^{2}\cdot\kappa\cdot\kappa_{\sigma^{2p}}}\nonumber\\
&\leq2\intmink{\kappa^{2}\kave_{\sigma^{2p}}^{2}}+\intmink{\norm{P_{4}^{4p,2p}\kave}}+\mathscr{L}^{-1}\intmink{\norm{P_{3}^{4p,2p}\kave}}.\label{LongTimeProp1,1}
\end{align}
Next, from the last result of Lemma \ref{AppendixLemma5} as well as Lemma \ref{AppendixLemma1},  we have
\begin{align}
&\intmink{\norm{P_{4}^{4p,2p}\kave}}\nonumber\\
&\leq c\mathscr{L}^{1-4\oo{p+1}}\oo{\kosc}^{1+\frac{1}{2\oo{2p+1}}}\oo{\mathscr{L}^{2\oo{2p+1}}\intmink{\kappa_{\sigma^{2p+1}}^{2}}}^{1-\frac{1}{2\oo{2p+1}}}\nonumber\\
&\leq c\mathscr{L}^{1-4\oo{p+1}}\kosc\oo{\mathscr{L}^{2\oo{2p+1}}\intmink{\kappa_{\sigma^{2p+1}}^{2}}}^{\frac{1}{2\oo{2p+1}}}\oo{\mathscr{L}^{2\oo{2p+1}}\intmink{\kappa_{\sigma^{2p+1}}^{2}}}^{1-\frac{1}{2\oo{2p+1}}}\nonumber\\
&=c\kosc\intmink{\kappa_{\sigma^{2p+1}}^{2}},\nonumber
\end{align}
as well as
\begin{align}
&\mathscr{L}^{-1}\intmink{\norm{P_{3}^{4p,2p}\kave}}\nonumber\\
&\leq\mathscr{L}^{1-4\oo{p+1}}\oo{\kosc}^{\frac{1}{2}+\frac{3}{4\oo{2p+1}}}\oo{\mathscr{L}^{2\oo{2p+1}}\intmink{\kappa_{\sigma^{2p+1}}^{2}}}^{1-\frac{3}{4\oo{2p+1}}}\nonumber\\
&\leq c\mathscr{L}^{1-4\oo{p+1}}\sqrt{\kosc}\oo{\mathscr{L}^{2\oo{2p+1}}\intmink{\kappa_{\sigma^{2p+1}}^{2}}}^{\frac{3}{4\oo{2p+1}}}\oo{\mathscr{L}^{2\oo{2p+1}}\intmink{\kappa_{\sigma^{2p+1}}^{2}}}^{1-\frac{3}{4\oo{2p+1}}}\nonumber\\
&=c\sqrt{\kosc}\intmink{\kappa_{\sigma^{2p+1}}^{2}}.\nonumber
\end{align}
Hence \eqref{LongTimeProp1,1} becomes
\begin{equation}
\frac{d}{dt}\intmink{\kappa_{\sigma^{p}}^{2}}+\oo{2\mathscr{Q}_{\star}-c\kosc-c\sqrt{\kosc}}\intmink{\kappa_{\sigma^{2p+1}}^{2}}\leq2\intmink{\kappa^{2}\kave_{\sigma^{2p}}^{2}}.\label{LongTimeProp1,2}
\end{equation}
The term on the right can be estimated easily by using our earlier $P$-style estimates, along with Lemma \ref{AppendixLemma0}:
\begin{align}
&2\intmink{\kappa^{2}\kave_{\sigma^{2p}}^{2}}\nonumber\\
&=2\intmink{\oo{\kave^{2}+\bar{\kappa}^{2}+2\bar{\kappa}\kave}\kave_{\sigma^{2p}}^{2}}\nonumber\\
&\leq 2\bar{\kappa}^{2}\intmink{\kappa_{\sigma^{2p}}^{2}}+\intmink{\norm{P_{4}^{4p,2p}\kave}}+\mathscr{L}^{-1}\intmink{\norm{P_{3}^{4p,2p}\kave}}\nonumber\\
&\leq\frac{8\mathscr{A}(\tilde{\mathcal{I}})^{2}}{\mathscr{L}^{2}}\intmink{\kappa_{\sigma^{2p}}^{2}}+c\oo{\kosc+\sqrt{\kosc}}\intmink{\kappa_{\sigma^{2p+1}}^{2}}\nonumber\\
&\leq\frac{8\mathscr{A}(\tilde{\mathcal{I}})^{2}}{\mathscr{L}^{2}}\oo{\varepsilon\mathscr{L}^{2}\intmink{\kappa_{\sigma^{2p+1}}^{2}}+\frac{1}{4\varepsilon^{2p}}\mathscr{L}^{-4p-1}\kosc}\nonumber\\
&\quad+c\oo{\kosc+\sqrt{\kosc}}\intmink{\kappa_{\sigma^{2p+1}}^{2}}\nonumber\\
&=8\mathscr{A}(\tilde{\mathcal{I}})^{2}\varepsilon\intmink{\kappa_{\sigma^{2p+1}}^{2}}+c\oo{\kosc+\sqrt{\kosc}}\intmink{\kappa_{\sigma^{2p+1}}^{2}}\nonumber\\
&\quad+\frac{4\mathscr{A}(\tilde{\mathcal{I}})^{2}}{2\varepsilon^{2p}}\mathscr{L}^{-4p-3}\kosc.\nonumber
\end{align}
Substituting into \eqref{LongTimeProp1,2}, we have
\begin{multline}
\frac{d}{dt}\intmink{\kappa_{\sigma^{p}}^{2}}+\oo{2\mathscr{Q}_{\star}-c\kosc-c\sqrt{\kosc}-8\mathscr{A}(\tilde{\mathcal{I}})^{2}\varepsilon}\intmink{\kappa_{\sigma^{2p+1}}^{2}}\\
\leq\frac{2\mathscr{A}(\tilde{\mathcal{I}})^{2}}{\varepsilon^{2p}}\mathscr{L}^{-4p-3}\kosc.\nonumber
\end{multline}
Choosing $\varepsilon$ sufficiently small and integrating over $\left[0,t\right)$ for $t\leq T$ gives
\[
\intmink{\kappa_{\sigma^{p}}^{2}}\Big|_{t}\leq c\int_{0}^{t}{\mathscr{L}^{-4p-3}\kosc\,d\tau}\leq c\oo{4\mathscr{A}\oo{\Gamma_{0}}\,\mathscr{A}(\tilde{\mathcal{I}})}^{-\frac{4p+3}{2}}\int_{0}^{t}{\kosc\,d\tau}\leq c\oo{\Gamma_{0},p}.
\]
Here we have used the isoperimetric inequality in the penultimate step. This completes the proof.
\end{proof}

\begin{thm}\label{Theorem2}
Suppose $\Gamma:\mathbb{S}^{1}\times\left[0,T\right)\rightarrow\mathcal{M}^{2}$ solves \eqref{APH} and that $\Gamma_{0}$ is simple with positive enclosed area. There exists a $\varepsilon_{0}>0$ (with $\varepsilon_{0}\leq\mathcal{K}^{\star}$) such that if
\[
\kosc\oo{\Gamma_{0}}<\varepsilon_{0}\,\,\text{and}\,\,\mathscr{I}\oo{\Gamma_{0}}<\exp\oo{\varepsilon_{0}/8\mathscr{A}(\tilde{\mathcal{I}})^{2}},
\]
then $\Gamma\oo{\mathbb{S}^{1}}$ approaches a homothetic rescaling of the isoperimetric $\partial\mathcal{U}$ with area equal to $\mathscr{A}\oo{\Gamma_{0}}$. 
\end{thm}
\begin{proof}
We begin by showing that $\kosc\searrow0$ as $t\nearrow\infty$. We will then discuss the ramification of this result. Recall from a previous discussion that to show $\kosc\searrow0$, it will be enough to show that $\norm{\kosc'}$ is bounded for all time. First by Theorem \ref{Theorem1} we know that for $\varepsilon_{0}>0$ sufficiently small, $T=\infty$. Moreover by Proposition \ref{Proposition2}, $\kosc\leq2\varepsilon_{0}$ for all time and by Proposition \ref{Proposition1} we have the estimate
\[
\norm{\frac{d}{dt}\kosc}\leq\oo{\frac{8\mathscr{A}(\tilde{\mathcal{I}})^{2}-\kosc}{\mathscr{L}}}\llll{\kappa_{\sigma^{p}}}_{2}^{2}\leq\frac{8\mathscr{A}(\tilde{\mathcal{I}})^{2}}{\sqrt{4\mathscr{A}\oo{\Gamma_{0}}\,\mathscr{A}(\tilde{\mathcal{I}})}}\llll{\kappa_{\sigma^{p}}}_{2}^{2}<c\oo{\Gamma_{0},p,\partial\mathcal{U}}.
\]
We have used the results from Proposition \ref{LongTimeProp1} in the last step, and the isoperimetric inequality in the penultimate step. This immediately tells us that $\kosc\searrow0$ as $t\nearrow\infty$. We will denote the limiting immersion by $\Gamma_{\infty}$. That is,
\[
\Gamma_{\infty}:=\lim_{t\to\infty}\Gamma_{t}\oo{\mathbb{S}^{1}}=\lim_{t\to\infty}\Gamma\oo{\cdot,t}.
\]
Our earlier equations imply that $\kosc\oo{\Gamma_{\infty}}\equiv0$. Note that because the isoperimetric inequality forces $\mathscr{L}\oo{\Gamma_{\infty}}\geq\sqrt{4 A\oo{\Gamma_{0}}\,\mathscr{A}(\tilde{\mathcal{I}})}>0$, we can not have $\mathscr{L}\searrow0$ and so we may conclude that
\begin{equation}
\int_{\Gamma_{\infty}}{\oo{\kappa-\bar{\kappa}}^{2}\,d\sigma}=0.\label{Theorem2,1}
\end{equation}
It follows that $\kappa(\Gamma_{\infty})\equiv C$ for some constant $C>0$.  That is to say, $\hat{k}\equiv C\oo{h+h_{\theta\theta}}^{-1}$ where $\hat{k}$ is the ordinary Euclidean curvature of $\Gamma_{\infty}$. Since by \eqref{EuclideanCurvature} the Euclidean curvature of the isoperimetrix is equal to $\oo{h+h_{\theta\theta}}^{-1}$, this imples that $\Gamma_{\infty}$ is a homothetic rescaling of the isoperimetrix $\tilde{I}$. Since the enclosed area does not change under the anisotropic curve diffusion flow, this homothetic rescaling indeed has enclosed area equal to $\mathscr{A}\oo{\Gamma_{0}}$.
\end{proof}
Since the previous theorem tells us that $\Gamma\oo{\mathbb{S}^{1},t}$ approaches a homothetical rescaling of $\partial\mathcal{U}$, we can conclude that for every $m\in\mathbb{N}$ there exists a sequence of times $\left\{t_{j}\right\}$ such that 
\[
\int{\kappa_{s^{m}}^{2}}\Big|_{t=t_{j}}\searrow0.
\]
Unfortunately, this is not the smooth classical notion of convergence that we desire, and does not allow us to rule out the possibility of short sharp ``spikes'' (oscillations) in time. Indeed, even if we were to show that for every $m\in\mathbb{N}$ we have $\llll{\kappa_{\sigma^{m}}}_{2}^{2}\in L^{1}\oo{\left[0,\infty\right)}$ (which is true), this would not be enough because these aforementioned ``spikes'' could occur on a time interval approaching that of (Lebesgue) measure zero. To overcome this dilemma, we attempt to control $\norm{\frac{d}{dt}\intmink{\kappa_{\sigma^{m}}^{2}}}$, and show that his quantity can be bounded by a multiple of $\kosc\oo{\Gamma_{0}}$ (which can be fixed to be as small as desired a priori). We will see that this allows us to strengthen the sequential convergence result above to a more classical exponential convergence (see Theorem). 

\begin{thm}[Exponential Convergence]\label{Theorem3}
Suppose $\Gamma:\mathbb{S}^{1}\times\left[0,T\right)\rightarrow\mathcal{M}^{2}$ solves \eqref{APH} as well as the assumptions of Theorem \ref{Theorem2}. Then for each $m\in\mathbb{N}\backslash\left\{0\right\}$ there is a time $t_{m}$ sufficiently large such that for $t\geq t_{m}$ there are constants $c_{m},c_{m}^{\star}$ with
\begin{equation}
\intmink{\kappa_{\sigma^{m}}^{2}}\leq c_{m}e^{-c_{m}^{\star}t},\label{Theorem3,0}
\end{equation}
and
\begin{equation}
\llll{\kappa_{\sigma^{m}}}_{\infty}^{2}\leq\oo{\mathscr{L}\oo{\Gamma_{0}}c_{m+1}/2\pi}e^{-c_{m+1}^{\star}t}.\label{Theorem3,00}
\end{equation}
\end{thm}
\begin{proof}
Using our earlier calculations from the proof of Proposition \ref{LongTimeProp1}, we have
\begin{align}
&\frac{d}{dt}\intmink{\kappa_{\sigma^{m}}^{2}}+2\mathscr{Q}_{\star}\intmink{\kappa_{m+p+1}^{2}}\nonumber\\
&\leq\sum_{j=1}^{\norm{p-m}}\sum_{i=1}^{j}\sum_{l=0}^{i-1}\bar{\kappa}^{-l}\intmink{\tilde{c}_{i}P_{i-l}^{j-i}\kave\cdot\kave_{\sigma^{m+p}}\kave_{\sigma^{m+p-j}}}\nonumber\\
&+\intmink{\norm{P_{4}^{2\oo{m+p},m+p}\kave}}+\mathscr{L}^{-1}\intmink{\norm{P_{3}^{2\oo{m+p},m+p}\kave}}\nonumber\\
&+2\oo{-1}^{\norm{m-p}}\bar{\kappa}^{2}\intmink{\kappa_{\sigma^{m+p}}^{2}}\nonumber\\
&\leq\sum_{j=1}^{\norm{p-m}}\sum_{i=1}^{j}\sum_{l=0}^{i-1}\mathscr{L}^{-l}\intmink{\norm{P_{i-l+2}^{2\oo{m+p}-j,m+p}\kave}}+\intmink{\norm{P_{4}^{2\oo{m+p},m+p}\kave}}\nonumber\\
&+\mathscr{L}^{-1}\intmink{\norm{P_{3}^{2\oo{m+p},m+p}\kave}}+2\oo{-1}^{\norm{m-p}}\bar{\kappa}^{2}\intmink{\kappa_{\sigma^{m+p}}^{2}}\nonumber\\
&\leq c\oo{\kosc+\sqrt{\kosc}}\intmink{\kappa_{\sigma^{m+p+1}}^{2}}+2\oo{-1}^{\norm{m-p}}\bar{\kappa}^{2}\intmink{\kappa_{\sigma^{m+p}}^{2}}.\label{Theorem3,1}
\end{align} 
Next we claim that for any smooth closed curve $\Gamma$ and any $l\in\mathbb{N}$ there exists a universal bounded constant $c_{l}$ such that
\begin{equation}
\intmink{\kappa_{\sigma^{l}}^{2}}\leq c_{l}\mathscr{L}^{2}\kosc\intmink{\kappa_{\sigma^{l+1}}^{2}}.\label{Theorem3,2}
\end{equation}
To prove this, we assume for the sake of contradiction that \eqref{Theorem3,2} is not true. Then there exists a series of immersions $\left\{\Gamma_{j}\right\}$ such that
\begin{equation}
\mathcal{R}_{j}:=\frac{\llll{\kappa_{\sigma^{l}}}_{2,\Gamma_{j}}^{2}}{\mathscr{L}^{2}\oo{\Gamma_{j}}\kosc\oo{\Gamma_{j}}\llll{\kappa_{\sigma^{l+1}}}_{2,\Gamma_{j}}^{2}}\nearrow\infty\,\,\text{as}\,\,j\nearrow\infty.\label{Theorem3,3}
\end{equation}
But by Lemma \ref{AppendixLemma1}, for each $j$ we have
\begin{equation}
\mathcal{R}_{j}\leq\frac{\frac{\mathscr{L}^{2}\oo{\Gamma_{j}}}{4\pi^{2}}\llll{\kappa_{\sigma^{l+1}}}_{2,\Gamma_{j}}^{2}}{\mathscr{L}^{2}\oo{\Gamma_{j}}\kosc\oo{\Gamma_{j}}\llll{\kappa_{\sigma^{l+1}}}_{2,\Gamma_{j}}^{2}}=\frac{1}{4\pi^{2}\kosc\oo{\Gamma_{j}}},\nonumber
\end{equation}
and so the only way for \eqref{Theorem3,3} to occur is if 
\begin{equation}
\kosc\oo{\Gamma_{j}}\searrow0\,\,\text{as}\,\,j\nearrow\infty.\label{Theorem3,4}
\end{equation} Then, as each $\Gamma_{j}$ satisfies the criteria of Theorem \ref{AppendixTheorem1}, we conclude there is a subsequence of immersions $\left\{\Gamma_{j_{k}}\right\}$ and an immersion $\Gamma_{\infty}$ such that $\Gamma_{j_{k}}\rightarrow\Gamma_{\infty}$ in the $C^{1}-$topology. Moreover, by \eqref{Theorem3,4}, we have $\kosc\oo{\Gamma_{\infty}}=0$. But this implies $\Gamma_{\infty}$ must be a homothetic rescaling of the isoperimetrix $\tilde{\mathcal{I}}$, in which case both sides on inequality \eqref{Theorem3,2} are zero. Hence the inequality holds trivially for the immersion $\Gamma_{\infty}$ with \emph{any} $c_{l}$ we wish, and so in fact we do not have $R_{j}\nearrow\infty$. This contradicts \eqref{Theorem3,3}, and we conclude that \eqref{Theorem3,2} must be true. Hence there exists a constant $c_{m}$ such that 
\[
\intmink{\kappa_{\sigma^{m}}^{2}}\leq c_{m}\mathscr{L}^{2}\kosc\intmink{\kappa_{\sigma^{m+1}}^{2}},
\]
and we conclude from \eqref{Theorem3,1} that 
\begin{equation}
\frac{d}{dt}\intmink{\kappa_{\sigma^{m}}^{2}}+\oo{2\mathscr{Q}_{\star}-c\oo{\kosc+\sqrt{\kosc}}}\intmink{\kappa_{\sigma^{m+2}}^{2}}\leq0.\nonumber
\end{equation}
for some universal constant $c$. Then, since $\kosc\searrow0$, there exists a time $t_{m}$ such that for $t\geq t_{m}$,
\begin{equation}
\frac{d}{dt}\intmink{\kappa_{\sigma^{m}}^{2}}\leq-\mathscr{Q}_{\star}\intmink{\kappa_{\sigma^{m+p+1}}^{2}}\leq-\mathscr{Q}_{\star}\oo{\frac{2\pi}{\mathscr{L}}}^{2\oo{p+1}}\intmink{\kappa_{\sigma^{m}}^{2}}.\label{Theorem3,5}
\end{equation}
Here we have used Lemma \ref{AppendixLemma1} $\oo{p+1}$ times. Integrating \eqref{Theorem3,5} over $\cc{t_{m},t}$ gives
\[
\intmink{\kappa_{\sigma^{m}}^{2}}\Big|_{t}\leq\oo{\intmink{\kappa_{\sigma^{m}}^{2}}\Big|_{t_{m}}\exp\oo{\mathscr{Q}_{\star}\oo{2\pi/\mathscr{L}}^{2\oo{p+1}}t_{m}}}\exp\oo{-\mathscr{Q}_{\star}\oo{2\pi/\mathscr{L}}^{2\oo{p+1}}t}.
\]
This is \eqref{Theorem3,0} with $c_{m}=\oo{\intmink{\kappa_{\sigma^{m}}^{2}}\Big|_{t_{m}}\exp\oo{\mathscr{Q}_{\star}\oo{2\pi/\mathscr{L}}^{2\oo{p+1}}t_{m}}}$ and\\
 $c_{m}^{\star}=\mathscr{Q}_{\star}\oo{2\pi/\mathscr{L}}^{2\oo{p+1}}$. Combining this with Lemma \ref{AppendixLemma2} then gives \eqref{Theorem3,00}. 
\end{proof}
Combining the results of Theorem \ref{Theorem2} and Theorem \ref{Theorem3} then proves our main result, Theorem \ref{MainTheorem1}.
\\\\
We finish this section with a proof of Proposition \ref{WaitingTimeProp}.
It is an adaptation of the proof of Proposition $1.5$ from \cite{Wheeler5} and Proposition $2$ from \cite{parkins2015polyharmonic}.
\begin{proof}[Proof of Proposition \ref{WaitingTimeProp}]
Recall that by the main theorem we have $T=\infty$. We may assume without loss of generality that there exists a time $t_{0}$ such that
\[
\begin{cases}
\kappa(\cdot,t)\ngtr0\,\,\text{for}\,\,t\in[0,t_{0}),\,\,\text{and}\\
\kappa(\cdot,t)>0\,\,\text{for}\,\,t\in[0,t_{0}).
\end{cases}
\]
We may also assume that
\begin{equation}
t_{0}>\frac{8\,\mathscr{A}(\tilde{\mathcal{I}})^{p-1}\mathscr{A}(\Gamma_{0})^{p+1}}{(p+1)\pi^{2p}}\oo{\mathscr{I}(\gamma_{0})^{p+1}-1},\label{WaitingTimeProp0}
\end{equation}
otherwise the proposition is trivially true. Next, our evolution equation for $\mathscr{L}$ from Corollary \ref{Corollary1} as well as Lemma \ref{AppendixLemma1} gives
\begin{equation}
\frac{d}{dt}\mathscr{L}=-\intmink{\kappa_{\sigma^{p}}^{2}}\leq-\oo{\frac{2\pi}{\mathscr{L}}}^{2(p-1)}\intmink{\kappa_{\sigma^{p}}^{2}}.\label{WaitingTimeProp1}
\end{equation}
Since by assumption for times $t\in[0,t_{0})$ there exists a point on the curve $\Gamma(t)$ with zero curvature, we may use Wirtinger's inequality (Lemma \ref{WirtingerInequality}) to obtain
\[
\intmink{\kappa^{2}}\leq \oo{\frac{\mathscr{L}}{\pi}}^{2}\intmink{\kappa_{\sigma}^{2}}
\]
and inserting this into \eqref{WaitingTimeProp1} gives the inequality
\begin{equation}
\frac{d}{dt}\mathscr{L}\leq-\oo{\frac{\pi}{\mathscr{L}}}^{2}\oo{\frac{2\pi}{\mathscr{L}}}^{2(p-1)}\intmink{\kappa^{2}}.\label{WaitingTimeProp2}
\end{equation}
Next, H\"older's inequality implies that
\[
\mathscr{A}(\tilde{\mathcal{I}})^{2}=\oo{\intmink{\kappa}}^{2}\leq\mathscr{L}\intmink{\kappa^{2}},
\]
so that \eqref{WaitingTimeProp2} becomes
\[
\frac{d}{dt}\mathscr{L}^{2(p+1)}\leq-2(p+1)\pi^{2}(2\pi)^{2(p-1)}A(\tilde{\mathcal{I}})^{2}.
\]
Integrating in time over $\cc{0,t_{0}}$, we find that
\[
\mathscr{L}^{2(p+1)}(t_{0})\leq\mathscr{L}^{2(p+1)}(0)-2(p+1)\pi^{2}(2\pi)^{2(p-1)}A(\tilde{\mathcal{I}})^{2}\,t_{0}.
\]
However, the choice of $t_{0}$ from \eqref{WaitingTimeProp0} implies that
\[
\mathscr{L}^{2(p+1)}(t_{0})<(4\mathscr{A}(\tilde{\mathcal{I}})\,\mathscr{A}(t_{0}))^{p+1},
\]
which contradicts the isoperimetric inequality. Therefore the proposition must be true.
\end{proof}
\end{section}

\begin{section}{Appendix}

\begin{proof}[Derivation of the tangent and normal vectors]
Let us now derive the equations $\eqref{MinkowskiTangentNormal}$ for  the unit tangent and normal vectors associated to an immersed curve in the Minkowski setting. These allow us to develop a Minkowski analogue of the Frenet-Serret equations in Section \ref{ConvexSection}. 
\\\\
Given an indicatrix as defined in the previous section, along with a curve $\Gamma$ with Euclidean tangent vector $\tau\oo{\theta}$, it is most logical to define the Minkowski unit tangent vector $T$ in the same direction as $\tau\oo{\theta}$ by
\[
T\oo{\theta}=r\oo{\theta}\tau\oo{\theta}.
\]
The Euclidean tangent has been multiplied by the radius of the indicatrix at the corresponding angle to insure $T$ is of Minkowski unit length. We wish to arrive at an analogue of the Frenet equations, and so wish to derive a set of equations in the form of 
\begin{equation}
\oo{
\begin{array}{c}
T\\
N
\end{array}
}_{\theta}
=
\oo{
\begin{array}{cc}
0 & \alpha\oo{\theta}\\
\beta\oo{\theta} & 0
\end{array}
}
\oo{
\begin{array}{c}
T\\
N
\end{array}
}\label{MinkowskiFrenet1}
\end{equation}
Presently, $\alpha,\beta$ are unknown functions, but we do know that because
$T$ and $N$ need to be $2\pi-$periodic that $\alpha,\beta$ will also need to be
$2\pi-$periodic.
Now we already know from the Euclidean Frenet-Serret equations,
$\eqref{MinkowskiFrenet1}$ and from our definition of $T$ that
\[
T_\theta=r_{\theta}\tau+r\tau_{\theta}=r_{\theta}\tau+rn=\alpha N.
\]
Differentiating this identity yields
\[
N_{\theta}=\frac{1}{r}\cc{\oo{\frac{r_{\theta}}{\alpha}}_{\theta}-\oo{\frac{r}{\alpha}}}T+\cc{\oo{\frac{r_{\theta}}{\alpha}}+\oo{\frac{r}{\alpha}}_{\theta}}n.
\]
We want $N_{\theta}$ to be solely in the direction of $\tau$ (and not $n$).
Hence the last equation forces
\[
0=\oo{\frac{r_{\theta}}{\alpha}}+\oo{\frac{r}{\alpha}}_{\theta}=\frac{2r_{\theta}}{\alpha}-\frac{r\alpha_{\theta}}{\alpha^{2}}\Leftrightarrow\frac{2r_{\theta}}{r}=\frac{\alpha_{\theta}}{\alpha}.
\]
This last equation is equivalent to
\[
2\oo{\ln{r}}_{\theta}=\oo{\ln{\alpha}}_{\theta},
\]
or
\[
\alpha=Cr^{2}
\]
for some constant $C$. Noting that $\det\oo{T,N}=C^{-1}$, we choose $C=1$ so
that the Minkowski area element is identical to its Euclidean counterpart.
Accordingly, the enclosed area $\mathscr{A}$ of a closed curve
$\Gamma:\mathbb{S}^{1}\rightarrow\mathcal{M}^{2}$ is simply equal to
\begin{equation}
\mathscr{A}\oo{\Gamma}=-\frac{1}{2}\intcurve{(\Gamma,n)}.\label{Area}
\end{equation}
Note that measure in the integral $ds$ could have been swapped for $d\sigma$.
We arrive at the following expression for the Minkowski tangent and normal vectors $T$ and $N$:
\begin{equation}
\oo{
\begin{array}{c}
T\\
N
\end{array}
}
=
\oo{
\begin{array}{cc}
r & 0\\
-h_{\theta} & h
\end{array}
}
\oo{
\begin{array}{c}
\tau\\
n
\end{array}
}\notag
\end{equation}
\end{proof}

\begin{prop}\label{TotalCurvatureProp}
Let $\Gamma:\mathbb{S}^{1}\rightarrow\mathcal{M}^{2}$ be a simple closed immersion in the Minkowski plane $\mathcal{M}$ with associated indicatrix $\partial\mathcal{U}$ and isoperimetrix $\tilde{\mathcal{I}}$. Then
\[
\intmink{\kappa}=2\,\mathscr{A}(\tilde{\mathcal{I}}),
\]
where $\mathscr{A}(\tilde{\mathcal{I}})$ denotes the enclosed area of the isoperimetrix.
\end{prop}
\begin{proof}
Using the identities $d\sigma=h\,ds$ and $d\theta=k\,ds$, we have
\begin{equation}
\intmink{\kappa}=\intmink{k(h+h_{\theta\theta})}=\int_{0}^{2\pi}{h(h+h_{\theta\theta})\,d\theta}.\label{TotalCurvatureProp1}
\end{equation}
Next, using the notation $\tau=(\cos\theta,\sin\theta),\,n=(-\sin\theta,\cos\theta)$, we can write the isoperimetrix $\tilde{I}(\theta)$ as
\[
\tilde{\mathcal{I}}(\theta)=\left\{-h_{\theta}\,\tau+h\,n\,:\,\theta\in[0,2\pi)\right\}.
\]
A quick calculation gives $\tilde{\mathcal{I}}_{\theta}=-(h+h_{\theta\theta})\,\tau$, which imples the induced Euclidean arc length and normal to $\tilde{\mathcal{I}}$ are given by
\[
d\tilde{s}=\sqrt{h+h_{\theta\theta}}\,d\theta\,\,\text{and}\,\,\tilde{n}=-\sqrt{h+h_{\theta\theta}}\,n,
\]
respectively. This implies that the signed enclosed area of $\tilde{\mathcal{I}}$ is given by
\begin{equation}
\mathcal{A}(\tilde{\mathcal{I}})=-\frac{1}{2}\intmink{(\tilde{\mathcal{I}},\tilde{n})\,d\tilde{s}}=\frac{1}{2}\int_{0}^{2\pi}{h(h+h_{\theta\theta})\,d\theta}.
\end{equation}
Comparing to
\[
\intmink{\kappa}=\intmink{k(h+h_{\theta\theta})}=\int_{0}^{2\pi}{h(h+h_{\theta\theta})\,d\theta},
\]
the results of the proposition then follow.
\end{proof}

\begin{lem}\label{AppendixLemma0}
Let $\Gamma:\mathbb{S}^{1}\rightarrow\mathcal{M}^{2}$ be a smooth closed curve with Minkowski curvature $\kappa$ and Minkowski arc length element $d\sigma$. Then for any $m\in\mathbb{N}$ we have
\[
\intmink{\kappa_{\sigma^{m}}^{2}}\leq\varepsilon \mathscr{L}^{2}\intmink{\kappa_{\sigma^{m+1}}^{2}}+\frac{1}{4\varepsilon^{m}}\mathscr{L}^{-\oo{2m+1}}K_{osc},
\]
for any $\varepsilon>0$.
\end{lem}
\begin{proof}
We will prove the lemma inductively. The case $m=1$ can be checked quite easily, by applying integration by parts and the Cauchy-Schwarz inequality:
\begin{align}
\intmink{\kappa_{\sigma}^{2}}&=\intmink{\oo{\kappa-\bar{\kappa}}_{\sigma}^{2}}\\
&=-\intmink{\oo{\kappa-\bar{\kappa}}\oo{\kappa-\bar{\kappa}}_{\sigma^{2}}}\nonumber\\
&\leq\oo{\intmink{\oo{\kappa-\bar{\kappa}}^{2}}}^{\frac{1}{2}}\oo{\intmink{\kappa_{\sigma^{2}}^{2}}}^{\frac{1}{2}}\nonumber\\
&\leq\varepsilon \mathscr{L}^{2}\intmink{\kappa_{\sigma^{2}}^{2}}+\frac{1}{4\varepsilon^{1}}\mathscr{L}^{-2}\intmink{\oo{\kappa-\bar{\kappa}}^{2}}.\nonumber
\end{align}
Next assume inductively that the statement is true for $j=m$. That is, assume that
\begin{equation}
\intmink{\kappa_{\sigma^{j}}^{2}}\leq\varepsilon \mathscr{L}^{2}\intmink{\kappa_{\sigma^{j+1}}^{2}}+\frac{1}{4\varepsilon^{j}}\mathscr{L}^{-\oo{2j+1}}K_{osc}\label{AppendixLemma0,1}
\end{equation}
for any $\varepsilon>0$.
Again performing integration by parts and the Cauchy-Schwarz inequality, we have for any $\varepsilon>0$:
\begin{align}
\intmink{\kappa_{\sigma^{j+1}}^{2}}&=-\intmink{\kappa_{\sigma^{j}}\cdot\kappa_{\sigma^{j+2}}}\\
&\leq\oo{\intmink{\kappa_{\sigma^{j}}^{2}}}^{\frac{1}{2}}\oo{\intmink{\kappa_{\sigma^j+2}^{2}}}^{\frac{1}{2}}\nonumber\\
&\leq\frac{\varepsilon}{2} \mathscr{L}^{2}\intmink{\kappa_{\sigma^{j+2}}^{2}}+\frac{1}{2\varepsilon}\mathscr{L}^{-2}\intmink{\kappa_{\sigma^{j}}^{2}}.\label{AppendixLemma0,2}
\end{align}
Substituting the inductive assumption \eqref{AppendixLemma0,1} into \eqref{AppendixLemma0,2} then gives
\begin{multline}
\intmink{\kappa_{\sigma^{j+1}}^{2}}\\
\leq\frac{\varepsilon}{2} \mathscr{L}^{2}\intmink{\kappa_{\sigma^{j+2}}^{2}}+\frac{1}{2\varepsilon}\mathscr{L}^{-2}\oo{\varepsilon \mathscr{L}^{2}\intmink{\kappa_{\sigma^{j+1}}^{2}}+\frac{1}{4\varepsilon^{j}}\oo{\varepsilon}\mathscr{L}^{-\oo{2j+1}}K_{osc}},\nonumber
\end{multline}
meaning that
\[
\frac{1}{2}\intmink{\kappa_{\sigma^{j+1}}^{2}}\leq\frac{\varepsilon}{2} \mathscr{L}^{2}\intmink{\kappa_{\sigma^{j+2}}^{2}}+\frac{1}{2}\cdot\frac{1}{4\varepsilon^{j+1}}\mathscr{L}^{-\oo{2\oo{j+1}+1}}K_{osc}.
\]
Multiplying out by $2$ then gives us the inductive step, completing the lemma.
\end{proof}

\begin{lem}\label{AppendixLemma1}
Let $f:\mathbb{R}\rightarrow\mathbb{R}$ be an absolutely continuous and periodic function of period $P$. Then, if $\int_{0}^{P}{f\,dx}=0$ we have
\[
\int_{0}^{P}{f^{2}\,dx}\leq\frac{P^{2}}{4\pi^{2}}\int_{0}^{P}{f_{x}^{2}\,dx},
\]
with equality if and only if
\[
f\oo{x}=A\cos\oo{\frac{2\pi}{P}x}+B\sin\oo{\frac{2\pi}{P}x}
\]
for some constants $A,B$.
\end{lem}
\begin{proof}
The proof is relatively straightforward, and follows from finding a function $f$ that extremises the integral $\int_{0}^{P}{f^{2}\,dx}$, given a constrained value of $\int_{0}^{P}{f_{x}^{2}\,dx}$, using the calculus of variations.
\end{proof}
\begin{lem}\label{AppendixLemma2}
Let $f:\mathbb{R}\rightarrow\mathbb{R}$ be an absolutely continuous and periodic function of period $P$. Then, if $\int_{0}^{P}{f\,dx}=0$ we have
\[
\llll{f}_{\infty}^{2}\leq\frac{P}{2\pi}\int_{0}^{P}{f_{x}^{2}\,dx}.
\]
\end{lem}
\begin{proof}
Since $\int_{0}^{P}{f\,dx}=0$ and $f$ is $P-$periodic we conclude that there exists distinct $0\leq p<q<P$ such that 
\[
f\oo{p}=f\oo{q}=0.
\]
Next, the fundamental theorem of calculus tells us that for any $x\in\oo{0,P}$,
\[
\frac{1}{2}\cc{f\oo{x}}^{2}=\int_{p}^{x}{ff_{u}\,du}=\int_{q}^{x}{ff_{u}\,du}.
\]
Hence
\begin{align*}
\oo{f\oo{x}}^{2}&=\int_{p}^{u}{ff_{u}\,du}-\int_{x}^{q}{ff_{u}\,du}\leq\int_{p}^{q}{\norm{ff_{x}}\,dx}\leq\int_{0}^{P}{\norm{ff_{x}}\,dx}\\
&\quad\leq\oo{\int_{0}^{P}{f^{2}\,dx}\cdot\int_{0}^{P}{f_{x}^{2}\,dx}}^{\frac{1}{2}}\leq\frac{P}{2\pi}\int_{0}^{P}{f_{x}^{2}\,dx},
\end{align*}
where the last step follows from Lemma \ref{AppendixLemma1}. We have also utilised H\"{o}lder's inequality with $p=q=2$.
\end{proof}

\begin{lem}[Wirtinger's inequality \cite{dym1985fourier}]\label{WirtingerInequality}
Let $f:\mathbb{R}\rightarrow\mathbb{R}$ be an absolutely continuous and periodic function of period $P$. If there exists a point $p\in\cc{0,P}$ with $f(p)=0$, then
\[
\int_{0}^{P}{f^{2}\,ds}\leq\oo{\frac{P}{\pi}}^{2}\int_{0}^{P}{f_{x}^{2}\,dx}.
\]
\end{lem}

\begin{lem}[Dziuk, Kuwert, and Sch\"{a}tzle \cite{Kuwert3}, Lemma $2.4$]\label{AppendixLemma3}
Let $\Gamma:\mathbb{S}^{1}\rightarrow\mathbb{R}^{2}$ be a smooth closed curve. Let $\phi:\mathbb{S}^{1}\rightarrow\mathbb{R}$ be a sufficiently smooth function. Then for any $l\geq 2,K\in\mathbb{N}$ and $0\leq i<K$ we have
\begin{equation}
\mathscr{L}^{i+1-\frac{1}{l}}\oo{\intmink{\oo{\phi}_{\sigma^{i}}^{2}}}^{\frac{1}{l}}\leq c\oo{K}\mathscr{L}^{\frac{1-\alpha}{2}}\oo{\intmink{\phi^{2}}}^{\frac{1-\alpha}{2}}\llll{\phi}_{K,2}^{\alpha}.\label{AppendixLemma3,1}
\end{equation}
Here $\alpha=\frac{i+\frac{1}{2}-\frac{1}{l}}{K}$, and
\[
\llll{\phi}_{K,2}:=\sum_{j=0}^{K}\mathscr{L}^{j+\frac{1}{2}}\oo{\intmink{\oo{\phi}_{\sigma^{j}}^{2}}}^{\frac{1}{2}}.
\]
In particular, if $\phi=\kappa-\bar{\phi}$, then
\begin{equation}
\mathscr{L}^{i+1-\frac{1}{l}}\oo{\intmink{\oo{k-\bar{k}}_{\sigma^{i}}^{2}}}^{\frac{1}{l}}\leq c\oo{K}\oo{K_{osc}}^{\frac{1-\alpha}{2}}\llll{k-\bar{k}}_{K,2}^{\alpha}.\label{AppendixLemma3,2}
\end{equation}
\end{lem}
\begin{proof}
The proof is identical to that of Lemma $2.4$ from \cite{Kuwert3} and is of a standard interpolative nature. Note that although we use $k-\bar{k}$ in the identity (as opposed to Kuwert et all who use $k$).
\end{proof}
\begin{lem}[Dziuk, Kuwert, and Sch\"{a}tzle \cite{Kuwert3}, Proposition $2.5$]\label{AppendixLemma4}
Let $\Gamma:\mathbb{S}^{1}\rightarrow\mathbb{R}^{2}$ be a smooth closed curve. Let $\phi:\mathbb{S}^{1}\rightarrow\mathbb{R}$ be a sufficiently smooth function. Then for any term $P_{\nu}^{\mu}\oo{\phi}$ (where $P_{\nu}^{\mu}\oo{\cdot}$ denotes the same $P$-style notation introduced in Section \ref{SectionIntroduction}) with $\nu\geq2$ which contains only derivatives of $\kappa$ of order at most $K-1$, we have
\begin{equation}
\intmink{\norm{P_{\nu}^{\mu}\oo{\phi}}}\leq c\oo{K,\mu,\nu}\mathscr{L}^{1-\mu-\nu}\oo{\mathscr{L}\intmink{\phi^{2}}}^{\frac{\nu-\eta}{2}}\llll{\phi}_{K,2}^{\eta}.\label{AppendixLemma4,1}
\end{equation}
In particular, for $\phi=\kappa-\bar{\kappa}$ we have the estimate
\begin{equation}
\intmink{\norm{P_{\nu}^{\mu}\oo{\kappa-\bar{\kappa}}}}\leq c\oo{K,\mu,\nu}\mathscr{L}^{1-\mu-\nu}\oo{K_{osc}}^{\frac{\nu-\eta}{2}}\llll{\kappa-\bar{\kappa}}_{K,2}^{\eta}\label{AppendixLemma4,2}
\end{equation}
where $\eta=\frac{\mu+\frac{\nu}{2}-1}{K}$.
\end{lem}
\begin{proof}
Using H\"older's inequality and Lemma \ref{AppendixLemma3} with $K=\nu$, if $\sum_{j=1}^{\nu}i_{j}=\mu$ we have
\begin{align}
\intmink{\norm{\phi_{\sigma^{i_{1}}}*\cdots*\phi_{\sigma^{i_{\nu}}}}}&\leq\prod_{j=1}^{\nu}\oo{\intmink{\phi_{\sigma^{i_{j}}}^{\nu}}}^{\frac{1}{\nu}}\\
&=\mathscr{L}^{1-\mu-\nu}\prod_{j=1}^{\nu}\mathscr{L}^{i_{j}+1-\frac{1}{\nu}}\oo{\intmink{\phi_{\sigma^{i_{j}}}^{\nu}}}^{\frac{1}{\nu}}\nonumber\\
&\leq c\oo{K,\mu,\nu}\mathscr{L}^{1-\mu-\nu}\prod_{j=1}^{\nu}\oo{\mathscr{L}\intmink{\phi^{2}}}^{\frac{1-\alpha_{j}}{2}}\llll{\phi}_{K,2}^{\alpha_{j}}\label{AppendixLemma4,3}
\end{align}
where $\alpha_{j}=\frac{i_{j}+\frac{1}{2}-\frac{1}{\nu}}{K}$. Now
\[
\sum_{j=1}^{\nu}\alpha_{j}=\frac{1}{K}\sum_{j=1}^{\nu}\oo{i_{j}+\frac{1}{2}-\frac{1}{\nu}}=\frac{\mu+\frac{\nu}{2}-1}{K}=\eta,
\]
and so substituting this into \eqref{AppendixLemma4,3} gives the first inequality of the lemma. It is then a simple matter of substituting $\phi=\kappa-\bar{\kappa}$ into this result to prove statement \eqref{AppendixLemma4,2}.
\end{proof}
\begin{lem}[Dziuk, Kuwert, and Sch\"{a}tzle \cite{Kuwert3}]\label{AppendixLemma5}
Let $\Gamma:\mathbb{S}^{1}\rightarrow\mathbb{R}^{2}$ be a smooth closed curve and $\phi:\mathbb{S}^{1}\rightarrow\mathbb{R}$ a sufficiently smooth function.
Then for any term $P_{\nu}^{\mu}\oo{\phi}$ with $\nu\geq2$ which contains only derivatives of $\kappa$ of order at most $K-1$, we have for any $\varepsilon>0$
\begin{multline}
\intmink{\norm{P_{\nu}^{\mu,K-1}\oo{\phi}}}\\
\leq c\oo{K,\mu,\nu}\mathscr{L}^{1-\mu-\nu}\oo{\mathscr{L}\intmink{\phi^{2}}}^{\frac{\nu-\eta}{2}}\oo{\mathscr{L}^{2K+1}\intmink{\phi_{\sigma^{K}}^{2}}+\mathscr{L}\intmink{\phi^{2}}}^{\frac{\eta}{2}}.\label{AppendixLemma5,1}
\end{multline}
Moreover if $\mu+\frac{1}{2}\nu<2K+1$ then $\eta<2$ and we have for any $\varepsilon>0$
\begin{equation}
\intmink{\norm{P_{\nu}^{\mu,K-1}\oo{\phi}}}\leq\varepsilon\intmink{\phi_{\sigma^{K}}^{2}}+c\cdot\varepsilon^{-\frac{\eta}{2-\eta}}\oo{\intmink{\phi^{2}}}^{\frac{\nu-\eta}{2-\eta}}+c\oo{\intmink{\phi^{2}}}^{\mu+\nu-1}.\label{AppendixLemma5,2}
\end{equation}
In particular, for $\phi=\kappa-\bar{\kappa}$, we have the estimate
\[
\intmink{\norm{P_{\nu}^{\mu}\oo{\kappa-\bar{\kappa}}}}\leq c\oo{K,\mu,\nu}\mathscr{L}^{1-\mu-\nu}\oo{\kosc}^{\frac{\nu-\eta}{2}}\oo{\mathscr{L}^{2K+1}\intmink{\oo{\kappa-\bar{\kappa}}_{\sigma^{K}}^{2}}}^{\frac{\eta}{2}}.
\]
Here, as before, $\eta=\frac{\mu+\frac{\nu}{2}-1}{K}$. 
\end{lem}
\begin{proof}
Combining the previous lemma with the following standard interpolation inequality from that follows from repeated applications of Lemma \ref{AppendixLemma0} (and is also found in \cite{Aubin1})
\[
\llll{\phi}_{K,2}^{2}\leq c\oo{K}\oo{\mathscr{L}^{2K+1}\intmink{\phi_{\sigma^{K}}^{2}}+\mathscr{L}\intmink{\phi^{2}}}
\]
yields the identity \eqref{AppendixLemma5,1} immediately. To prove \eqref{AppendixLemma5,2} we simply combine \eqref{AppendixLemma5,1} with the Cauchy-Schwarz identity. The final identity of the Lemma follow by letting $\phi=\kappa-\bar{\kappa}$ in \eqref{AppendixLemma5,1} and combining this with the identity
\begin{equation}
K_{osc}\leq \mathscr{L}\oo{\frac{\mathscr{L}^{2}}{4\pi^{2}}}^{K}\intmink{\oo{\kappa-\bar{\kappa}}_{\sigma^{K}}^{2}}= c\oo{K}\mathscr{L}^{2K+1}\intmink{\oo{\kappa-\bar{\kappa}}_{\sigma^{K}}^{2}},\label{AppendixLemma5,3}
\end{equation}
which is a direct consequence of applying Lemma \ref{AppendixLemma1} $\oo{p+1}$ times repeatedly.
\end{proof}

\begin{lem}\label{AppendixLemma4}
Let $\Gamma:\mathbb{S}^{1}\rightarrow\mathcal{M}^{2}$ be a closed Minkowski curve with Minkowski curvature $\kappa$. Then for any $m,L\in\mathbb{N}$ with $m<L$ we have the estimate
\[
\intmink{\kappa_{\sigma^{m}}^{2}}\leq\oo{\intmink{\oo{\kappa-\bar{\kappa}}^{2}}}^{1-\frac{m}{L}}\oo{\intmink{\kappa_{\sigma^{L}}^{2}}}^{\frac{m}{L}}.
\]
\end{lem}
\begin{proof}
First we prove the intermediate identity
\begin{equation}
\intmink{\kappa_{\sigma^{m}}^{2}}\leq\oo{\intmink{\oo{\kappa-\bar{\kappa}}^{2}}}^{\frac{1}{m+1}}\oo{\intmink{\kappa_{\sigma^{m+1}}^{2}}}^{\frac{m}{m+1}},\,\,m\in\mathbb{N}.\label{AppendixLemma4,1}
\end{equation}
For $m=1$ we simply use integration by parts and H\"older's inequality:
\begin{align*}
\intmink{\kappa_{\sigma}^{2}}&=-\intmink{\oo{\kappa-\bar{\kappa}}\kappa_{\sigma\sigma}}\leq\oo{\intmink{\oo{\kappa-\bar{\kappa}}^{2}}}^{\frac{1}{2}}\oo{\intmink{\kappa_{\sigma\sigma}^{2}}}^{\frac{1}{2}}\\
&=\oo{\intmink{\oo{\kappa-\bar{\kappa}}^{2}}}^{\frac{1}{1+1}}\oo{\intmink{\kappa_{\sigma\sigma}^{2}}}^{\frac{1}{1+1}}.
\end{align*}
Next we inductively assume that \eqref{AppendixLemma4,1} is true for $m=1,2,\dots,M$. Using integration by parts and H\"older's inequality once more gives
\begin{align}
&\intmink{\kappa_{\sigma^{M+1}}^{2}}=-\intmink{\kappa_{\sigma^{M}}\kappa_{\sigma^{M+1}}}\leq\oo{\intmink{\kappa_{\sigma^{M}}^{2}}}^{\frac{1}{2}}\oo{\intmink{\kappa_{\sigma^{M+1}}^{2}}}^{\frac{1}{2}}\nonumber\\
&\quad\leq\oo{\intmink{\oo{\kappa-\bar{\kappa}}^{2}}}^{\frac{1}{2\oo{M+1}}}\oo{\intmink{\kappa_{\sigma^{M+1}}^{2}}}^{\frac{M}{2\oo{M+1}}}\oo{\intmink{\kappa_{\sigma^{M+2}}^{2}}}^{\frac{1}{2}}\nonumber\\
&\quad\leq\frac{M}{2\oo{M+1}}\intmink{\kappa_{\sigma^{M+1}}^{2}}+\frac{M+2}{2\oo{M+1}}\oo{\intmink{\oo{\kappa-\bar{\kappa}}^{2}}}^{\frac{1}{M+2}}\oo{\intmink{\kappa_{\sigma^{M+2}}^{2}}}^{\frac{M+1}{M+2}}.\nonumber
\end{align}
Absorbing gives the statement \eqref{AppendixLemma4,1} for $M+1$, completing the induction. Here we have used our inductive assumption in the second line, and H\"older's inequality with $p=\frac{2\oo{M+1}}{M},q=\frac{2\oo{M+1}}{M+2}$ in the last line. To arrive at the claim of the lemma, we simply employ the inequality \eqref{AppendixLemma4,1} repeatedly:
\begin{align}
\intmink{\kappa_{\sigma^{m}}^{2}}&\leq\oo{\intmink{\oo{\kappa-\bar{\kappa}}^{2}}}^{\frac{1}{m+1}}\oo{\intmink{\kappa_{\sigma^{m+1}}^{2}}}^{\frac{m}{m+1}}\nonumber\\
&\leq\oo{\intmink{\oo{\kappa-\bar{\kappa}}^{2}}}^{\frac{1}{m+1}}\oo{\oo{\intmink{\oo{\kappa-\bar{\kappa}}^{2}}}^{\frac{1}{m+2}}\oo{\intmink{\kappa_{\sigma^{m+2}}^{2}}}^{\frac{m+1}{m+2}}}^{\frac{m}{m+1}}\nonumber\\
&=\oo{\intmink{\oo{\kappa-\bar{\kappa}}^{2}}}^{\frac{m}{m\oo{m+1}}+\frac{m}{\oo{m+1}\oo{m+2}}}\oo{\intmink{\kappa_{\sigma^{m+2}}^{2}}}^{\frac{m}{m+2}}\nonumber\\
&\vdots\nonumber\\
&\leq\oo{\intmink{\oo{\kappa-\bar{\kappa}}^{2}}}^{m\sum_{i=0}^{L-m-1}\left\{\frac{1}{\oo{m+i}\oo{m+i+1}}\right\}}\oo{\intmink{\kappa_{\sigma^{L}}^{2}}}^{\frac{m}{L}}.\label{AppendixLemma4,2}
\end{align}
Now $\sum_{i=0}^{L-m-1}\left\{\frac{1}{\oo{m+i}\oo{m+i+1}}\right\}$ is a telescoping series, which sums to $1/m-1/L$. Substituting this value into \eqref{AppendixLemma4,2} then proves the lemma.
\end{proof}

\begin{thm}[\cite{Breuning1}, Theorem $1.1$]\label{AppendixTheorem1}
Let $q\in\mathbb{R}^{n}$, $m,p\in\mathbb{N}$ with $p>m$. Additionally, let $\mathcal{A},\mathcal{V}>0$ be some fixed constants. Let $\mathfrak{T}$ be the set of all mappings $f:\Sigma:\rightarrow\mathbb{R}^{n}$ with the following properties:
\begin{itemize}
\item $\Sigma$ is an $m$-dimensional, compact manifold (without boundary)

\item $f$ is an immersion in $W^{2,p}\oo{\Sigma,\mathbb{R}^{n}}$ satisfying
\begin{align*}
\llll{A\oo{f}}_{p}&\leq\mathcal{A},\\
\text{vol}\oo{\Sigma}&\leq\mathcal{V},\text{  and  }\\
q&\in f\oo{\Sigma}.
\end{align*}
\end{itemize}
Then for every sequence $f^{i}:\Sigma^{i}\rightarrow\mathbb{R}^{n}$ in $\mathfrak{T}$ there is a subsequence $f^{j}$, a mapping $f:\Sigma\rightarrow\mathbb{R}^{n}$ in $\mathfrak{T}$ and a sequence of diffeomorphisms $\phi^{j}:\Sigma\rightarrow\Sigma^{j}$ such that $f^{j}\circ\phi^{j}$ converges in the $C^{1}$-topology to $f$.
\end{thm}

\end{section}

\bibliographystyle{plain}

\bibliography{ThesisBibliography}

\begin{thebibliography}{10}

\bibitem{Abresch1}
Uwe Abresch and Joel Langer.
\newblock The normalized curve shortening flow and homothetic solutions.
\newblock {\em Journal of Differential Geometry}, 23(2):175--196, 1986.

\bibitem{Altschuler1}
Steven~J Altschuler.
\newblock {\em Singularities of the curve shrinking flow for space curves}.
\newblock PhD thesis, University of California, San Diego, Department of
  Mathematics, 1990.

\bibitem{Aubin1}
Thierry Aubin.
\newblock {\em Nonlinear analysis on manifolds: Monge-Ampere equations}, volume
  252.
\newblock Springer, 1982.

\bibitem{Breuning1}
Patrick Breuning.
\newblock Immersions with bounded second fundamental form.
\newblock {\em arXiv preprint arXiv:1201.4562}, 2012.

\bibitem{chakerian1960isoperimetric}
G.D. Chakerian.
\newblock The isoperimetric problem in the {M}inkowski plane.
\newblock {\em American Mathematical Monthly}, pages 1002--1004, 1960.

\bibitem{do1976differential}
Manfredo~Perdigao Do~Carmo and Manfredo~Perdigao Do~Carmo.
\newblock {\em Differential geometry of curves and surfaces}, volume~2.
\newblock Prentice-hall Englewood Cliffs, 1976.

\bibitem{dym1985fourier}
Harry Dym, Henry~P McKean, David Aldous, and Yung~L Tong.
\newblock Fourier series and integrals.
\newblock 1985.

\bibitem{Kuwert3}
Gerhard Dziuk, Ernst Kuwert, and Reiner Schatzle.
\newblock Evolution of elastic curves in $\mathbb{R}^{n}$: Existence and
  computation.
\newblock {\em SIAM journal on mathematical analysis}, 33(5):1228--1245, 2002.

\bibitem{edwards2014shrinking}
Maureen Edwards, Alexander Gerhardt-Bourke, James McCoy, Glen Wheeler, and
  Valentina-Mira Wheeler.
\newblock The shrinking figure eight and other solitons for the curve diffusion
  flow.
\newblock {\em Journal of Elasticity}, 119(1-2):191--211, 2014.

\bibitem{Hamilton2}
Michael Gage and Richard~S Hamilton.
\newblock The heat equation shrinking convex plane curves.
\newblock {\em J. Differential Geom}, 23(1):69--96, 1986.

\bibitem{gage1983isoperimetric}
Michael~E Gage.
\newblock An isoperimetric inequality with applications to curve shortening.
\newblock {\em Duke Math. J}, 50(4):1225--1229, 1983.

\bibitem{Gage2}
Michael~E Gage.
\newblock Curve shortening makes convex curves circular.
\newblock {\em Inventiones mathematicae}, 76(2):357--364, 1984.

\bibitem{Gage1}
Michael~E Gage and Yi~Li.
\newblock Evolving plane curves by curvature in relative geometries.
\newblock In {\em Duke Math. J}. Citeseer, 1993.

\bibitem{grayson1987heat}
Matthew~A Grayson.
\newblock The heat equation shrinks embedded plane curves to round points.
\newblock {\em Journal of Differential geometry}, 26(2):285--314, 1987.

\bibitem{martini2004geometry}
Horst Martini and Konrad~J Swanepoel.
\newblock The geometry of {M}inkowski spaces-a survey. part ii.
\newblock {\em Expositiones mathematicae}, 22(2):93--144, 2004.

\bibitem{martini2001geometry}
Horst Martini, Konrad~J Swanepoel, and Gunter Wei{\ss}.
\newblock The geometry of {M}inkowski spaces-a survey. part i.
\newblock {\em Expositiones mathematicae}, 19(2):97--142, 2001.

\bibitem{parkins2015polyharmonic}
Scott Parkins and Glen Wheeler.
\newblock The polyharmonic heat flow of closed plane curves.
\newblock {\em arXiv preprint arXiv:1505.02877}, 2015.

\bibitem{pozzi2012gradient}
Paola Pozzi.
\newblock On the gradient flow for the anisotropic area functional.
\newblock {\em Mathematische Nachrichten}, 285(5-6):707--726, 2012.

\bibitem{Wheeler5}
Glen Wheeler.
\newblock On the curve diffusion flow of closed plane curves.
\newblock {\em Annali di Matematica Pura ed Applicata}, 192(5):931--950, 2013.

\end{thebibliography}

\end{document}